%% file: NFMT.tex
\documentclass[
	11pt,
	a4paper,
	appendixprefix,
	openany,
	bibtotoc,
	twoside=semi,
	DIV=15,
	abstracton,
	tocindent,
	pointlessnumbers,
	bigheadings,
]{scrartcl}

\setkomafont{sectioning}{\normalfont\normalcolor\bfseries}
\setkomafont{title}{\normalfont\normalcolor\bfseries}

\usepackage[english]{babel}
\usepackage[T1]{fontenc}
\usepackage[utf8]{inputenc}
\usepackage{ae,aecompl}
\usepackage{microtype}
\usepackage{csquotes}
\usepackage{calc}
\usepackage{mdwlist}
\usepackage{ltablex}
\usepackage{ragged2e}
\newcolumntype{L}[1]{>{\raggedright\arraybackslash}p{#1}}
\newcolumntype{C}[1]{>{\centering\arraybackslash}p{#1}}
\newcolumntype{R}[1]{>{\raggedleft\arraybackslash}p{#1}}
\newcolumntype{J}[1]{>{\justifying\arraybackslash}p{#1}}
\usepackage{multirow}
\usepackage{multicol}
\usepackage{graphicx}
\usepackage{nicefrac}
\usepackage{amsmath}
\usepackage{amssymb}
\usepackage{lmodern}
\usepackage{mathtools}
\usepackage{caption}
\usepackage{subcaption}
\usepackage{aligned-overset}
\usepackage{textcmds}

\usepackage{tikz}
\usepackage{pgfplots}
\usepackage{tikz-3dplot}
\pgfplotsset{/pgf/number format/use comma} 
\usepgfplotslibrary{statistics}
\usepgfplotslibrary{polar}
\usepackage[pdfborder={0 0 0}]{hyperref}
\usepackage{changes}

\usepackage[linesnumbered]{algorithm2e}
\SetKwInput{KwCc}{Computational cost}
\SetAlgoSkip{bigskip}

\setkomafont{pagehead}{\small\scshape}
\fussy

\usepackage{amsthm}
\newtheorem{thm}{Theorem}[section]
\newtheorem{lem}[thm]{Lemma}

\theoremstyle{definition}
\newtheorem{defi}[thm]{Definition}
\theoremstyle{remark}
\newtheorem*{expl}{Example}
\newtheorem*{rem}{Remark}

\usepackage[backend=biber,style=trad-abbrv,sorting=nyt,url=false,doi=false,eprint=false]{biblatex}
\numberwithin{equation}{section}

\title{ANOVA approximation with mixed tensor product basis on scattered points}
\author{
  Daniel Potts\footnotemark[1]\and
  Pascal Schröter\footnotemark[3]
}

\begin{document}
\input{makros.tex}

\allowdisplaybreaks
\newcommand{\bas}[3]{\phi_{#1}^{{#2},{#3}}}
\newcommand{\cosbas}[2]{\bas{#1}{#2}{\cos}}
\newcommand{\expbas}[2]{\bas{#1}{#2}{\exp}}
\newcommand{\xcos}{\mathbf{x}^{\cos}}
\newcommand{\xexp}{\mathbf{x}^{\exp}}
\newcommand{\kcos}{\kbf^{\cos}}
\newcommand{\kexp}{\kbf^{\exp}}
\newcommand{\lcos}{\mathbf{l}^{\cos}}
\newcommand{\lexp}{\mathbf{l}^{\exp}}
\newcommand{\alg}{\mop{alg}}
\newcommand{\expbf}{\mathbf{exp}}
\newcommand{\cosbf}{\mathbf{cos}}
\newcommand{\meng}[3]{\Ical_{#1}^{{#2},{#3}}}
\newcommand{\mentil}[3]{\tilde{\Ical}_{#1}^{{#2},{#3}}}
\newcommand{\defb}[2]{\Dbb^{{#1},{#2}}}
\newcommand{\nfft}{\mathsf{NFFT}}
\newcommand{\nfftt}{{\mathsf{NFFT}^\top}}
\newcommand{\nffct}{\mathsf{NFMT}}
\newcommand{\dfct}{\mathsf{DFCT}}
\newcommand{\nffctt}{{\mathsf{NFMT}^\top}}
\newcommand{\nfct}{\mathsf{NFCT}}
\dedication{This Paper is dedicated to the 70th birthday of Albrecht Böttcher}
\maketitle
\begin{abstract}
	In this paper we consider an orthonormal basis, generated by a tensor product of Fourier basis functions, half period cosine basis functions, and the Chebyshev basis functions. We deal with the approximation problem in high dimensions related to this basis and design a fast algorithm to multiply with the underlying matrix, consisting of rows of the non-equidistant Fourier matrix, the non-equidistant cosine matrix and the non-equidistant Chebyshev matrix, and its transposed. Using this, we derive the ANOVA (analysis of variance) decomposition for functions with partially periodic boundary conditions through using the Fourier basis in some dimensions and the half period cosine basis or the Chebyshev basis in others. We consider sensitivity analysis in this setting, in order to find an adapted basis for the underlying approximation problem. More precisely, we find the underlying index set of the multidimensional series expansion. Additionally, we test this ANOVA approximation with mixed basis at numerical experiments, and refer to the advantage of interpretable results.
	\subsubsection*{Key words}
	ANOVA, high-dimensional approximation, Fourier approximation, fast Fourier methods, NFFT, Chebyshev polynomials
	\subsubsection*{AMS subject classifications}
	65T, 42B05
\end{abstract}

\footnotetext[1]{potts@math.tu-chemnitz.de, Chemnitz University of
	Technology, Faculty of Mathematics, D--09107 Chemnitz, Germany}

\footnotetext[3]{pascal.schroeter@math.tu-chemnitz.de, Chemnitz University of
	Technology, Faculty of Mathematics, D--09107 Chemnitz, Germany}

\section{Introduction}\label{sec:1}
The approximation of functions is a problem that arises in many scientific fields. As soon as data is recorded, questions how \ldq{}Which correlations are in the data?\rdq{}, \ldq{}Which variables are dependent on one another\rdq{} and \ldq{}How data can be predicted at other points?\rdq{} arises.\\
To formalise these questions we introduce the following. Let $\Xcal$ be a set of typically high-dimensional data points which we refer as nodes. Let $f$ be a continuous function that maps each node $\xbf$ contained in $\Xcal$ to a value $f(\xbf)$. Now the task is to find a model that approximates the function $f$, which we can easily evaluate at arbitrary nodes. The second requirement for the model is that it should be possible to find out with reasonable effort which combinations of the dimensions of the nodes influence the values of the model in which way.\\
There are various algorithms as artificial neural networks or support vector machines for approximating functions in high dimensions, see e.g. \cite{Agg14,elements}. However, these only reveal the hidden dependencies in the data if they are examined in detail, which is costly. There are methods that are particularly well suited to answering the question of hidden connections in the data. We focus on the ANOVA (analysis of variance) decomposition, cf. \cite{CaMoOw97,RaAl99,LiOw06,KuSloWaWo10,Holtz11,Gu2013} and \cite[Appendix A]{Owe13}. This method gave really good results in this task as it is shown in \cite{PoSc22, schmischkediss}. We build our theory on the foundation provided by this work.\\
The classical ANOVA is based on an integral projection operator. 
We use consequently the series expansion in various basic functions to define the ANOVA decomposition, cf. \cite{PoSchmi19,PoSc19b,schmischkediss}. A majority of real world systems are dominated by low-complexity interactions of their variables. This principle is known as sparsity-of-effects, see e.g. \cite[Section 4.6]{WuHa21},\cite{HaSchShToTrWa22},\cite[Section 4.2]{schmischkediss}. We use this principle to truncate the ANOVA decomposition. 
Since we consider a finite truncation of the series expansion for approximation, the properties of the basis functions, such as the periodicity, are reflected in the approximation. It is therefore advantageous if the basis functions have similar properties as the data of the underlying process. 
We will combine well known basis functions on $[0,1]$ like the Fourier basis functions $\phi_k^{\exp}\coloneqq\exp(2\pi\im k \cdot)$, $k\in\Zbb$, the half period cosine basis functions $\phi_k^{\cos}\coloneqq\sqrt{2}^{\;1-\delta_{k,0}}\cos(\pi k \cdot)$, $k\in\Nbb_0$, and the Chebyshev basis functions $\phi_k^{\alg}\coloneqq\sqrt{2}^{\;1-\delta_{k,0}}\cos(k \arccos(2\cdot-1))$, $k\in\Nbb_0$ in a tensor product structure to achieve more flexibility what properties are present in which dimensions. We denote this tensor product basis functions by $\phi_\kbf^\dbf\coloneqq\prod_{j=1}^{d}\phi_{k_j}^{d_j}$. Here $\dbf$ is a vector containing the information which basis is used in which dimension. We assume that it is known from the application which base in which dimension should be used.
In a nutshell, we use a finite sum of these basis functions to approximate a function $f$, i.e.
\begin{align*}
	\sum_{\kbf\in\Ical}\hat{f}_\kbf\phi_\kbf\approx f.
\end{align*}
In this way we get a model of the data with which we can predict further data. In order to get good approximations, it is essential to use appropriate index sets $\Ical$ in the approximating sums. The choice of such an index set is always a trade-off between the number of indices, the number of training data available, and the needed computation time.\\
Furthermore, we use another concept of the classical ANOVA, namely the analytic global sensitivity indices \cite{So90,So01,LiOw06}. These tell us which variables are related and how big the influence of these relations are. Using the coefficients $\hat{f}_\kbf$ from our model for the function $f$, we can calculate approximated global sensitivity indices. Using these approximated global sensitivity indices, we can truncate our approximation even further to get better suitable index sets $\Ical$, which provides us even better approximations.\\
To compute such approximations it is important to evaluate finite sums of basis functions $\sum_{\kbf\in\Ical}\hat{f}_\kbf\phi_\kbf(\xbf)$ with known index sets $\Ical$ at many scattered nodes $\xbf_j\in[0,1]^d$ simultaneously. A main focus of this work is the development of algorithms to evaluate these sums of mixed basis functions $\phi_\kbf^\dbf$. If we are dealing with the Fourier basis in each dimension, the resulting sums can be evaluated through algorithms like the non-equidistant fast Fourier transform ($\nfft$) \cite{KeKuPo09},\cite[Chapter 7]{PlPoStTa18} combined with grouped transforms, cf. \cite{BaPoSchmi21,PoSc19b}. Similar things work, if we deal with the cosine basis in each dimension or the Chebyshev basis in each dimension. We will develop an fast algorithm based on the $\nfft$ to evaluate these sums at many scattered nodes.\\
The reason, why we have to evaluate these sums at scattered nodes is, that real world data is rarely equidistantly sampled. The usage of real world data provides another challenge, because it is typically not in the domain $[0,1]^d$. This problem can be solved by rescaling the data, e.g. with a min-max normalisation. If the source of the data is not bounded, we need to use a different transformation. For example, if the data is normally distributed, they can be pre-processed with the error function to bring them into the domain $[0,1]^d$, see e.g. \cite{PoSc22}.\\
The paper is organized as follows. In Section \ref{sec:2} we set up notion and terminology. Firstly, in Subsection \ref{sec:2.1} we introduce needed function spaces and some of their relations. Subsection \ref{sec:2.2} introduces some well known orthonormal basis and finally the mixed basis with which we will deal with in the rest of the paper. In Section \ref{sec:3} we introduce the ANOVA approximation for the mixed basis based on an approach by Fourier series. Here, Subsection \ref{sec:3.1} deals with the definition of the ANOVA decomposition and Subsection \ref{sec:3.2} provides a way to compute an ANOVA approximation. We split this subsection in three parts, firstly we consider useful index sets then we describe a way to compute the ANOVA approximation for given index sets and to this end we describe how this index sets could be determined. In Section \ref{sec:4} we develop fast algorithms to evaluate sums of mixed basis functions. Subsection \ref{sec:4.1} provides through Theorem \ref{s:1} a way to compute such sums through an $\nfft$ which is summarized in Algorithm \ref{alg:1}. In Subsection \ref{sec:4.2} we extend the grouped transform \cite{BaPoSchmi21} to the mixed basis using the Algorithm \ref{alg:1}. In Section \ref{sec:5} we show with some experiments how this approximation procedure works. Subsection \ref{sec:5.1} deals with the approximation of a function. There are the steps shown to find good suitable index sets. In Subsection \ref{sec:5.2} we approximate a function where we only have access to uniformly sampled nodes. At this point we compare the approximation with a suitable mixed basis with the approximation with the half period cosine basis and the Fourier basis. Subsection \ref{sec:5.3} presents an approximation on a publicly available benchmark dataset.

\section{Preliminaries}\label{sec:2}
This section presents basic definitions for the rest of this paper. In Subsection \ref{sec:2.1} various function spaces and some of their relations are introduced. We use \cite{schmischkediss} as reference for this. We use these as foundations for our later considerations of the ANOVA decomposition. Subsection \ref{sec:2.2} presents some orthonormal basis. We start by defining some classical one-dimensional basis functions. In the next step, Definition \ref{d:1}, we combine these basis functions for higher dimensions

\subsection{Function Spaces}\label{sec:2.1}
Let $\Dbb\subset\menk{\Tbb^m\times[0,1]^n}{m,n\in\Nbb_0}$ be a measurable set, where $\Tbb=\Rbb/\Zbb$ is the torus which we identify with $[0,1)$. Moreover, let $\fktk{\omega}{\Dbb}{(0,\infty)}$ be a probability measure with $\igl{\Dbb}{\omega(\xbf)}{\xbf}=1$. We define the weighted Lebesgue spaces
\begin{align*}
	\Lp{p}(\Dbb,\omega)\coloneqq\men{\fktk{f}{\Dbb}{\Cbb}}{\igl{\Dbb}{\betr{f(\xbf)}^p\omega(\xbf)}{\xbf}<\infty}
\end{align*}
with the norm $\norm{f}_{\Lp{p}(\Dbb,\omega)}\coloneqq(\igl{\Dbb}{\betr{f(\xbf)}^p\omega(\xbf)}{\xbf})^{\frac{1}{p}}$ for $p\in[1,\infty)$. Furthermore, we define
\begin{align*}
	\Lp{\infty}(\Dbb)\coloneqq\menk{\fktk{f}{\Dbb}{\Cbb}}{\esssup_{\xbf\in\Dbb}\betr{f(\xbf)}<\infty}
\end{align*}
with the norm $\normk{f}_{\Lp{\infty}(\Dbb)}\coloneqq{}{\esssup}_{\xbf\in\Dbb}\betr{f(\xbf)}$. The Lebesgue space $\Lp{2}(\Dbb,\omega)$ forms a Hilbert space with the scalar product $\skpr{f}{g}{\Lp{2}(\Dbb,\omega)}\coloneqq\igl{\Dbb}{f(\xbf)\overline{g(\xbf)}\omega(\xbf)}{\xbf}$. We use the abbreviation $\Lp{p}(\Dbb)\coloneqq\Lp{p}(\Dbb,1)$. Let $\Bcal=(\phi_k)_{k\in\Kcal}$ be a basis with an index set $\Kcal$ for the Hilbert space $\Lp{2}(\Dbb,\omega)$. We define the Wiener space 
\begin{align*}
	\Acal(\Dbb,\omega)\coloneqq\men{f\in\Lp{1}(\Dbb,\omega)}{\sum_{k\in\Kcal}\betr{\skpr{f}{\phi_k}{\Lp{2}(\Dbb,\omega)}}<\infty}
\end{align*}
with the norm $\normk{f}_{\Acal(\Dbb,\omega)}\coloneqq\sum_{k\in\Kcal}\betrk{\skpr{f}{\phi_k}{\Lp{2}(\Dbb,\omega)}}$.
\begin{lem}\label{s:10}
	Let $\Lp{2}(\Dbb,\omega)$ be a weighted Lebesgue space with a basis $(\phi_k)_{k\in\Kcal}$ with $\sup_{k\in\Kcal}\normk{\phi_k}_{\Lp{\infty}(\Dbb)}<\infty$. Then every element of the corresponding Wiener space $\Acal(\Dbb,\omega)$ has a continuous representative.
\end{lem}
\begin{proof}
	See \cite[Lemma 2.5]{schmischkediss}.
\end{proof}
\begin{rem}
	Lemma \ref{s:10} provides $\Acal(\Dbb,\omega)\subseteq C(\Dbb)\coloneqq\menk{\fktk{f}{\Dbb}{\Rbb}}{f\text{ continuous}}$. Furthermore, we get
	\begin{align*}
		\norm{f}_{\Lp{\infty}(\Dbb)}&=\esssup_{x\in\Dbb}\betr{\sum_{k\in\Kcal}\skpr{f}{\phi_k}{\Lp{2}(\Dbb,\omega)}\phi_k(x)}\\
		&\leq\sup_{k\in\Kcal}\normk{\phi_k}_{\Lp{\infty}(\Dbb)}\sum_{k\in\Kcal}\betr{\skpr{f}{\phi_k}{\Lp{2}(\Dbb,\omega)}}\\
		&=\sup_{k\in\Kcal}\normk{\phi_k}_{\Lp{\infty}(\Dbb)}\norm{f}_{\Acal(\Dbb,\omega)}.
	\end{align*}
	It follows $\Acal(\Dbb,\omega)\subseteq\Lp{\infty}(\Dbb)$ and $\Acal(\Dbb,\omega)\subseteq \Lp{2}(\Dbb,\omega)$.
\end{rem}
Due to this Lemma, we define the evaluation of a function $f\in\Acal(\Dbb,\omega)$ at a point $\xbf\in\Dbb$ as evaluation of the continuous representative at the point $\xbf$. Next, we consider partial sums of the function $f$ for finite subsets of the index set $\Ical\subset\Kcal$, e.g. $S_\Ical(\Bcal) f\coloneqq\sum_{k\in\Ical}\skpr{f}{\phi_k}{\Lp{2}(\Dbb,\omega)}\phi_k$. Furthermore, we define the set of polynomials related to the finite index set $\Ical$ as
\begin{align}
	\Tcal_\Ical(\Bcal)\coloneqq\men{\sum_{k\in\Ical}c_k\phi_k}{c_k\in\Cbb}\label{eq:30}.
\end{align}

\subsection{Orthonormal Basis}\label{sec:2.2}
In this subsection we firstly introduce the one-dimensional basis functions which we use for the mixed basis.
The functions $\phi_k^{\exp}=\exp(2\pi \im k \cdot)$ form the orthonormal Fourier basis of $\Lp{2}(\Tbb)$. Additionally, the functions $\phi_k^{\cos}=\sqrt{2}\cos(\pi k \cdot)$ form together with the constant function with value one the orthonormal half period cosine basis of $\Lp{2}([0,1])$. The functions $\phi_k^{\alg}=\sqrt{2}\cos(k \arccos(2\cdot-1))$ form together with the constant function with value one the orthonormal Chebyshev basis of $\Lp{2}([0,1],\omega)$ with the weight $\fkt{\omega}{[0,1]}{(0,\infty)}{x}{\frac{1}{\pi\sqrt{x-x^2}}}$. In the following we are going to work with tensor products of these basis functions.
\begin{defi}\label{d:1}
	Let $\dbf$ be a $d$-dimensional tuple over the set $\{\exp,\cos,\alg\}$. We define the sets 
    \begin{align*}
        \Dbb^\dbf\coloneqq\bigtimes_{j=1}^d
        \begin{cases}
            \Tbb,&d_j=\exp\\
            [0,1],&d_j\neq\exp\\
        \end{cases}&&\text{ and }&&\Kbb^\dbf\coloneqq\bigtimes_{j=1}^d
        \begin{cases}
            \Zbb,&d_j=\exp\\
            \Nbb_0,&d_j\neq\exp
        \end{cases}
    \end{align*}
    and the mixed functions 
    \begin{align*}
        \fkt{\phi_\kbf^\dbf}{\Dbb^\dbf}{\Cbb}{\xbf}{\prod_{j=1}^{d}
        \begin{cases}
            1,&k_j=0\\
            \exp(2\pi\im k_j x_j),&d_j=\exp,\;k_j\neq0\\
            \sqrt{2}\cos(\pi k_j x_j),&d_j=\cos,\;k_j\neq0\\
			\sqrt{2}\cos(k_j \arccos(2x_j-1)),&d_j=\alg,\;k_j\neq0
        \end{cases}}
    \end{align*}
    for $\kbf\in\Kbb^\dbf$. Furthermore, we define the weight function 
	\begin{align*}
		\fkt{\omega^\dbf}{\Dbb^\dbf}{(0,\infty)}{\xbf}{\prod_{j=1}^{d}
        \begin{cases}
            1,&d_j\neq \alg\\
			\frac{1}{\pi\sqrt{x_j-x_j^2}},&d_j=\alg
        \end{cases}}\quad.
	\end{align*}
\end{defi}
The mixed functions $\phi_\kbf^\dbf$ form a basis of $\Lp{2}(\Dbb^\dbf,\omega^\dbf)$ because of their tensor product structure and because their factors are Fourier, half period cosine and Chebyshev basis functions. We name this basis $\Bcal^\dbf\coloneqq\menk{\phi_\kbf^\dbf}{\kbf\in\Kbb^\dbf}$.

\section{Interpretable ANOVA Approximation}\label{sec:3}
In this section, we define the ANOVA approximation for the mixed basis. We follow the steps in \cite{PoSc19b}. We do this by writing a function as a series expansion and splitting it into parts. We than define sensitivity indices for the parts to determine how important they are and use them to truncate the series expansion. We start in Subsection \ref{sec:3.1} with defining the ANOVA decomposition, see \cite{CaMoOw97,LiOw06,KuSloWaWo10,Gu2013} and \cite[Appendix A]{Owe13}, in the way like \cite{PoSchmi19,PoSc19b} through a series expansion. We apply this to our setting with the mixed basis. Furthermore, we define analytic global sensitivity indices \cite{So90,So01,LiOw06}. In Subsection \ref{sec:3.2} we describe the procedure of ANOVA approximation \cite{PoSchmi19,PoSc19b,PoSc21}, and we deduce a way to compute it numerically.

\subsection{ANOVA Decomposition}\label{sec:3.1}
Let $f$ be an $\Lp{2}(\Dbb^\dbf,\omega^\dbf)$ function. Since $\phi_\kbf^\dbf$ with $\kbf\in\Kbb^\dbf$ form an orthonormal basis of $\Lp{2}(\Dbb^\dbf,\omega^\dbf)$, $f$ can be written as
\begin{align*}
	f=\sum_{\kbf\in\Kbb^\dbf}c_\kbf^\dbf(f)\phi_\kbf^\dbf
\end{align*}
with coefficients $c^\dbf_\kbf(f)\coloneqq\skprk{f}{\phi_\kbf^\dbf}{\Lp{2}(\Dbb^\dbf,\omega^\dbf)}$. Furthermore, we get the Parseval equality
\begin{align*}
	\norm{f}_{\Lp{2}(\Dbb^\dbf,\omega^\dbf)}^2=\sum_{\kbf\in\Kbb^\dbf}\betrk{c^{\dbf}_\kbf(f)}^2,
\end{align*}
from the fact that $\Bcal^\dbf$ is a basis of $\Lp{2}(\Dbb^\dbf,\omega^\dbf)$.
Next, we decompose the function $f$ into ANOVA terms. We denote subsets of coordinate indices with small boldface letters $\ubf\in\Pcal([d])$. For every subset of indices $\ubf$ we define the ANOVA term
\begin{align*}
	f_{\ubf}(\xbf)\coloneqq\sum_{\substack{\kbf\in\Kbb^\dbf\\\supp\kbf=\ubf}}c^{\dbf}_\kbf(f)\phi_\kbf^\dbf(\xbf).
\end{align*}
Note that such an ANOVA term $f_\ubf(\xbf)$ is independent of $x_j$ if $j\notin\ubf$. We point out, that ANOVA terms $f_\ubf$ of a function $f$ may be smoother than the function $f$ itself, see \cite{LiOw06,GrKuSl10,schmischkediss}. These ANOVA terms $f_\ubf$, $\ubf\subseteq[d]$ decompose the function $f$ uniquely into 
\begin{align*}
	f=\sum_{\kbf\in\Kbb^\dbf}c^{\dbf}_\kbf(f)\phi_\kbf^\dbf=\sum_{\ubf\in\Pcal([d])}f_\ubf.
\end{align*}
This follows since $\menk{\menk{\kbf\in\Kbb^\dbf}{\supp\kbf=\ubf}}{\ubf\in\Pcal([d])}$ is a partition of the set $\Kbb^\dbf$. Additionally, we define the variance of a function $f$ as $\sigma^2(f)=\igl{\Dbb^\dbf}{\betrk{f(\xbf)-c_0^\dbf(f)}^2\omega^\dbf(\xbf)}{\xbf}$, which is equivalent to $\sigma^2(f)=\normk{f}_{\Lp{2}(\Dbb^\dbf,\omega^\dbf)}^2-\betrk{c^{\dbf}_\nullbf(f)}^2$. The Parseval equality states
\begin{align*}
	\sigma^2(f)=\sum_{\kbf\in\Kbb^\dbf\setminus\{\nullbf\}}\betrk{c^{\dbf}_\kbf(f)}^2.
\end{align*}
Furthermore, we get the variance of ANOVA terms $f_\ubf$ through
\begin{align*}
	\sigma^2(f_\ubf) = \sum_{\kbf\in\Kbb^\dbf\setminus\{\nullbf\}}\betrk{c^{\dbf}_\kbf(f_\ubf)}^2 = \sum_{\substack{\kbf\in\Kbb^\dbf\\\supp\kbf=\ubf}}\betrk{c^{\dbf}_\kbf(f_\ubf)}^2=\norm{f_\ubf}_{\Lp{2}(\Dbb^\dbf,\omega^\dbf)}^2.
\end{align*}
Finally, we define the analytic global sensitivity indices (GSI) like \cite{So01} through
\begin{align*}
	\rho(\ubf,f)\coloneqq\frac{\sigma^2(f_\ubf)}{\sigma^2(f)}.
\end{align*}
The analytic GSI's have values in $[0,1]$, and the closer an analytic GSI $\rho(\ubf,f)$ is to 1, the more important is the corresponding ANOVA term $f_\ubf$ for reconstructing the function $f$. We use this information for the construction of good suitable index sets. We point out that the analytic GSIs depend on the weight $\omega^\dbf$. Next we truncate the ANOVA decomposition. We use a set of subsets of indices $U\subseteq\Pcal([d])$ for this truncation. We define 
\begin{align*}
	\Trm_Uf(\xbf)=\sum_{\ubf\in U}f_\ubf(\xbf).
\end{align*}
To find this set $U$, we choose the ANOVA terms $f_\ubf$ with the highest GSI's to get $\Trm_Uf\approx f$. In order to do this, we choose a threshold $\theta\in[0,1)$ and set $U=\menk{\ubf\subset[d]}{\rho(\ubf,f)>\theta}$.

\subsection{Numerical ANOVA Approximation}\label{sec:3.2}
In this section our aim is to approximate a function $f\in\Acal(\Dbb^\dbf,\omega^\dbf)$. We are given a set of $M\in\Nbb$ nodes $\Xcal\subset\Dbb^\dbf,\;\betr{\Xcal}=M$ and the corresponding function values $(f(\xbf))_{\xbf\in\Xcal}\coloneqq\fbf\in\Cbb^M$. We aim to find a mixed polynomial 
\begin{align*}
	\fkt{f^\dbf}{\Dbb^\dbf}{\Cbb}{\xbf}{\sum_{\kbf\in\Ical}\hat{f}_\kbf^\dbf\phi_\kbf^\dbf(\xbf)},\quad\hat{f}^\dbf_\kbf\in\Cbb
\end{align*}
for which $f\approx f^\dbf$ holds, i.e. $\normk{f-f^\dbf}_{\Lp{2}(\Dbb^\dbf,\omega^\dbf)}$ is small and $\Ical$ is a finite subset of $\Kbb^\dbf$. In the next Subsection \ref{sec:3.2.1} we consider some ways to choose $\Ical$. The next Subsection \ref{sec:3.2.2} presents a way to find the mixed polynomial $f^\dbf$ which minimizes the $\Lp{2}(\Dbb^\dbf,\omega^\dbf)$ norm of $f-f^\dbf$ for a given index set $\Ical$ and nodes $\Xcal$. To this end we show in Subsection \ref{sec:3.2.3} how to choose and refine the truncation set $U$.

\subsubsection{Grouped Index Sets}\label{sec:3.2.1}
We present some index sets that are important for this paper. Since these index sets contain frequencies for the mixed basis we call them frequency sets. We begin with frequency sets which are full $d$-dimensional hypercubes, but since their size grows exponential in the dimension $d$ we introduce better controllable frequency sets. We start with the frequency sets which are full $d$-dimensional hypercubes, i.e.
\begin{align}\label{eq:1}
	\Ical_\Nbf^\dbf\coloneqq\bigtimes_{j=1}^d
	\begin{cases}
		\Zbb\cap\left[-\frac{N_j}{2},\frac{N_j}{2}\right),&d_j=\exp\\
		\Nbb_0\cap\left[0,\frac{N_j}{2}\right),&d_j\neq\exp
	\end{cases}
\end{align}
for a vector of bandwidths $\Nbf=(N_j)_{j=1}^d\in(2\Nbb_0)^d$. These sets have the cardinality
\begin{align}
	\betrk{\Ical_\Nbf^\dbf}=\prod_{j=1}^{d}
	\begin{cases}
		N_j,&d_j=\exp\\
		\frac{N_j}{2},&d_j\neq\exp
	\end{cases}\label{eq:12}.
\end{align}
Next, we define frequency sets that can be adjusted more precisely. We use them to construct thinner frequency sets. In detail these frequency sets should have a high bandwidth along the coordinate axes, less bandwidth along the coordinate planes and so on. For this purpose we define the following frequency sets $\tilde{\Ical}_\Nbf^\dbf$ with bandwidths $\Nbf=(N_j)_{j=1}^d\in(2\Nbb_0)^{d}$. Here if $N_j$ is zero the set $\tilde{\Ical}_\Nbf^\dbf$ should only contain the frequency zero, when it is projected onto the dimension $j$. If $N_j$ is not zero the projection onto the dimension $j$ should not contain the frequency zero, because it is contained in a lower dimensional set. This is being done by the frequency set
\begin{align}
	\tilde{\Ical}_\Nbf^\dbf\coloneqq\bigtimes_{j=1}^d
	\begin{cases}
		\{0\},&N_j=0\\
		\Zbb\cap\left[-\frac{N_j}{2},\frac{N_j}{2}\right)\setminus\{0\},&d_j=\exp\text{ and }N_j\neq0\\
		\Nbb_0\cap\left[0,\frac{N_j}{2}\right)\setminus\{0\},&d_j\neq\exp\text{ and }N_j\neq0
	\end{cases}.\label{eq:9}
\end{align}
These frequency sets  $\tilde{\Ical}_\Nbf^\dbf$ have the cardinality
\begin{align}
	\betrk{\tilde{\Ical}_\Nbf^\dbf}=\prod_{j=1}^d
	\begin{cases}
		1,&N_j=0\\
		N_j-1,&d_j=\exp\text{ and }N_j\neq0\\
		\frac{N_j}{2}-1,&d_j\neq\exp\text{ and }N_j\neq0
	\end{cases}\label{eq:20}.
\end{align}
Since these frequency sets are disjoint, if the bandwidths have different support, we can form the union of them to derive a new frequency set. We choose for every $\ubf\in U$ a bandwidth $\Nbf^\ubf=(N_j^\ubf)_{j=1}^{d}\in(2\Nbb)^d$ with $N_j^\ubf\neq0$ for $j\in \ubf$ and $N_j^\ubf=0$ for $j\notin \ubf$ and define the frequency set
\begin{align}
	\Ical(U)\coloneqq\bigcup_{\ubf\in U}\tilde{\Ical}_{\Nbf^\ubf}^\dbf.\label{eq:10}
\end{align}
At this point we have a look at special sets U. We choose a superposition dimension $d_s<d$ and define the superposition set $U_{d_s}$ which contains only subsets $\ubf$ of the size up to $\betrk{\ubf}\leq d_s\in\Nbb$, i.e.
\begin{align}
	U_{d_s}\coloneqq\menk{\ubf\subseteq[d]}{\betrk{\ubf}\leq d_s}\label{eq:21}.
\end{align}
In other words, we focus on the interactions of $d_s$ or less dimensions. Mixed polynomials $\Tcal_{\Ical(U_{d_s})}(\Bcal^\dbf)$ consist of basis functions which have only up to the superposition dimension $d_s$ many interactions between the variables $x_j$. Approximating functions with such basis functions of only low-order interactions is a common problem, see \cite{Holtz11,schmischkediss,XieShiSchWa22,LiPoUl21}. Next we consider the cardinality of the superposition set $U_{d_s}$.
\begin{lem}\label{s:2}
Let $d_s$ be the superposition dimension and $U_{d_s}$ the corresponding superposition set, see \eqref{eq:21}. Furthermore let the entries of $\Nbf^\ubf$ be bounded for all $\ubf\in U_{d_s}$, i.e. $\max_{\substack{\ubf\in U_{d_s}\\i\in[d]}}N_i^\ubf\leq N_{\max}$. Then 
\begin{align}
	\betr{\Ical(U_{d_s})}\leq d_s(dN_{\max})^{d_s}\label{eq:31}
\end{align}
holds for $d\geq2d_s$.
\end{lem}
\begin{proof}
	We have a look at the cardinality of the frequency set $\Ical(U_{d_s})$ and using the fact that the sets $\tilde{\Ical}_{\Nbf^\ubf}^\dbf$ are disjoint we get
	$\betrk{\Ical(U)}=\betrk{\bigcup_{\ubf\in U_{d_s}}}=\sum_{\ubf\in U_{d_s}}\betrk{\tilde{\Ical}_{\Nbf^\ubf}^\dbf}$. In the first part we show $\betrk{\tilde{\Ical}_{\Nbf^\ubf}^\dbf}\leq (N_{\max})^{d_s}$ and in the second part $\betrk{U_{d_s}}\leq d_sd^{d_s}$.\\
	Since $\ubf\in U_{d_s}$ has at most $d_s$ entries, $\Nbf^\ubf$ has at most $d_s$ non-zero entries and using \eqref{eq:20} we get $\betrk{\tilde{\Ical}_{\Nbf^\ubf}^\dbf}\leq (N_{\max})^{d_s}$.\\
	We can estimate the cardinality of the truncation set $U_{d_s}$ as 
	\begin{align*}
		\betr{U_{d_s}}\leq\betr{\menk{\ubf\subset[d]}{\betrk{\ubf}\leq d_s}}=\sum_{i=1}^{d_s}\binom{d}{i}\leq d_s\binom{d}{d_s}\leq d_sd^{d_s}.
	\end{align*}
	This gives us $\betrk{\Ical(U_{d_s})}=\sum_{\ubf\in U_{d_s}}\betrk{\tilde{\Ical}_{\Nbf^\ubf}^\dbf}\leq\sum_{\ubf\in U_{d_s}}(N_{\max})^{d_s}\leq d_sd^{d_s}(N_{\max})^{d_s}=d_s(dN_{\max})^{d_s}$
\end{proof}
The Equation \eqref{eq:31} shows that the frequency set $\Ical(U_{d_s})$, with a superposition set $U_{d_s}$ containing only subsets $\ubf$ of the size up to $\betrk{\ubf}\leq d_s\in\Nbb$, grows only polynomially with the power $d_s$, which is an improvement over the exponentially growing full $d$-dimensional hypercubes $\Ical_\Nbf^\dbf$.
\begin{figure}
	\centering
	\tdplotsetmaincoords{76}{100}
	\begin{tikzpicture}
		[tdplot_main_coords,scale=0.66]
		\def\yl{0}
		\foreach \y in {-9,-8,...,-6}{
			\foreach \z in {0}{
				\draw[fill={rgb,255:red,0; green,255; blue,0},rounded corners=0.001] (\yl+0.5, \y-0.5, \z-0.5) -- (\yl+0.5, \y+0.5, \z-0.5) -- (\yl+0.5, \y+0.5, \z+0.5) -- (\yl+0.5, \y-0.5, \z+0.5) -- cycle;
			}
		}
		\foreach \y in {0}{
			\foreach \z in {4}{
				\draw[fill={rgb,255:red,255; green,0; blue,0},rounded corners=0.001] (\yl+0.5, \y-0.5, \z-0.5) -- (\yl+0.5, \y+0.5, \z-0.5) -- (\yl+0.5, \y+0.5, \z+0.5) -- (\yl+0.5, \y-0.5, \z+0.5) -- cycle;
			}
		}
		\foreach \y in {5,6,...,8}{
			\foreach \z in {0}{
				\draw[fill={rgb,255:red,0; green,255; blue,0},rounded corners=0.001] (\yl+0.5, \y-0.5, \z-0.5) -- (\yl+0.5, \y+0.5, \z-0.5) -- (\yl+0.5, \y+0.5, \z+0.5) -- (\yl+0.5, \y-0.5, \z+0.5) -- cycle;
			}
		}
		\def\yl{3}
		\foreach \y in {-5,-4,...,-1}{
			\foreach \z in {0}{
				\draw[fill={rgb,255:red,0; green,255; blue,255},rounded corners=0.001] (\yl+0.5, \y-0.5, \z-0.5) -- (\yl+0.5, \y+0.5, \z-0.5) -- (\yl+0.5, \y+0.5, \z+0.5) -- (\yl+0.5, \y-0.5, \z+0.5) -- cycle;
			}
		}
		\foreach \y in {1,2,...,4}{
			\foreach \z in {0}{
				\draw[fill={rgb,255:red,0; green,255; blue,255},rounded corners=0.001] (\yl+0.5, \y-0.5, \z-0.5) -- (\yl+0.5, \y+0.5, \z-0.5) -- (\yl+0.5, \y+0.5, \z+0.5) -- (\yl+0.5, \y-0.5, \z+0.5) -- cycle;
			}
		}
		\def\yl{7}
		\foreach \y in {0}{
			\foreach \z in {0}{
				\draw[fill={rgb,255:red,0; green,0; blue,255},rounded corners=0.001] (\yl+0.5, \y-0.5, \z-0.5) -- (\yl+0.5, \y+0.5, \z-0.5) -- (\yl+0.5, \y+0.5, \z+0.5) -- (\yl+0.5, \y-0.5, \z+0.5) -- cycle;
			}
		}

		\def\zl{0}
		\foreach \x in {0}{
			\foreach \y in {-9,-8,...,-1}{
				\draw[fill={rgb,255:red,127; green,255; blue,127},rounded corners=0.001] (\x-0.5, \y-0.5, \zl+0.5) -- (\x+0.5, \y-0.5, \zl+0.5) -- (\x+0.5, \y+0.5, \zl+0.5) -- (\x-0.5, \y+0.5, \zl+0.5) -- cycle;
			}
		}
		\foreach \x in {1,2,...,3}{
			\foreach \y in {-5,-4,...,-1}{
				\draw[fill={rgb,255:red,127; green,255; blue,225},rounded corners=0.001] (\x-0.5, \y-0.5, \zl+0.5) -- (\x+0.5, \y-0.5, \zl+0.5) -- (\x+0.5, \y+0.5, \zl+0.5) -- (\x-0.5, \y+0.5, \zl+0.5) -- cycle;
			}
		}
		\foreach \x in {3,4,...,7}{
			\foreach \y in {0}{
				\draw[fill={rgb,255:red,127; green,127; blue,225},rounded corners=0.001] (\x-0.5, \y-0.5, \zl+0.5) -- (\x+0.5, \y-0.5, \zl+0.5) -- (\x+0.5, \y+0.5, \zl+0.5) -- (\x-0.5, \y+0.5, \zl+0.5) -- cycle;
			}
		}
		\foreach \x in {1,2,...,3}{
			\foreach \y in {1,2,...,4}{
				\draw[fill={rgb,255:red,127; green,255; blue,255},rounded corners=0.001] (\x-0.5, \y-0.5, \zl+0.5) -- (\x+0.5, \y-0.5, \zl+0.5) -- (\x+0.5, \y+0.5, \zl+0.5) -- (\x-0.5, \y+0.5, \zl+0.5) -- cycle;
			}
		}
		\foreach \x in {0}{
			\foreach \y in {1,2,...,8}{
				\draw[fill={rgb,255:red,127; green,255; blue,127},rounded corners=0.001] (\x-0.5, \y-0.5, \zl+0.5) -- (\x+0.5, \y-0.5, \zl+0.5) -- (\x+0.5, \y+0.5, \zl+0.5) -- (\x-0.5, \y+0.5, \zl+0.5) -- cycle;
			}
		}
		\def\zl{4}
		\foreach \x in {0}{
			\foreach \y in {0}{
				\draw[fill={rgb,255:red,225; green,127; blue,127},rounded corners=0.001] (\x-0.5, \y-0.5, \zl+0.5) -- (\x+0.5, \y-0.5, \zl+0.5) -- (\x+0.5, \y+0.5, \zl+0.5) -- (\x-0.5, \y+0.5, \zl+0.5) -- cycle;
			}
		}

		\def\xl{0}
		\foreach \x in {4,5,...,7}{
			\foreach \z in {0}{
				\draw[fill={rgb,255:red,0; green,0; blue,127},rounded corners=0.001] (\x-0.5, \xl+0.5, \z-0.5) -- (\x+0.5, \xl+0.5, \z-0.5) -- (\x+0.5, \xl+0.5, \z+0.5) -- (\x-0.5, \xl+0.5, \z+0.5) -- cycle;
			}
		}
		\foreach \x in {0}{
			\foreach \z in {1,2,...,4}{
				\draw[fill={rgb,255:red,127; green,0; blue,0},rounded corners=0.001] (\x-0.5, \xl+0.5, \z-0.5) -- (\x+0.5, \xl+0.5, \z-0.5) -- (\x+0.5, \xl+0.5, \z+0.5) -- (\x-0.5, \xl+0.5, \z+0.5) -- cycle;
			}
		}
		\def\xl{4}
		\foreach \x in {1,2,...,3}{
			\foreach \z in {0}{
				\draw[fill={rgb,255:red,0; green,127; blue,127},rounded corners=0.001] (\x-0.5, \xl+0.5, \z-0.5) -- (\x+0.5, \xl+0.5, \z-0.5) -- (\x+0.5, \xl+0.5, \z+0.5) -- (\x-0.5, \xl+0.5, \z+0.5) -- cycle;
			}
		}
		\def\xl{8}
		\foreach \x in {0}{
			\foreach \z in {0}{
				\draw[fill={rgb,255:red,0; green,127; blue,0},rounded corners=0.001] (\x-0.5, \xl+0.5, \z-0.5) -- (\x+0.5, \xl+0.5, \z-0.5) -- (\x+0.5, \xl+0.5, \z+0.5) -- (\x-0.5, \xl+0.5, \z+0.5) -- cycle;
			}
		}

		\def\yl{2}
		\foreach \y in {0}{
			\foreach \z in {1,2,...,3}{
				\draw[fill={rgb,255:red,255; green,0; blue,255},rounded corners=0.001] (\yl+0.5, \y-0.5, \z-0.5) -- (\yl+0.5, \y+0.5, \z-0.5) -- (\yl+0.5, \y+0.5, \z+0.5) -- (\yl+0.5, \y-0.5, \z+0.5) -- cycle;
			}
		}
		\def\zl{3}
		\foreach \x in {1,2,...,2}{
			\foreach \y in {0}{
				\draw[fill={rgb,255:red,225; green,127; blue,225},rounded corners=0.001] (\x-0.5, \y-0.5, \zl+0.5) -- (\x+0.5, \y-0.5, \zl+0.5) -- (\x+0.5, \y+0.5, \zl+0.5) -- (\x-0.5, \y+0.5, \zl+0.5) -- cycle;
			}
		}
		\def\xl{0}
		\foreach \x in {1,2,...,2}{
			\foreach \z in {1,2,...,3}{
				\draw[fill={rgb,255:red,127; green,0; blue,127},rounded corners=0.001] (\x-0.5, \xl+0.5, \z-0.5) -- (\x+0.5, \xl+0.5, \z-0.5) -- (\x+0.5, \xl+0.5, \z+0.5) -- (\x-0.5, \xl+0.5, \z+0.5) -- cycle;
			}
		}

	\draw[->,>=latex,very thick] (7.5,0,0) -- (11,0,0) node[anchor=north]{$x_2$};
	\draw[->,>=latex,very thick] (0,8.5,0) -- (0,9,0) node[anchor=north]{$x_1$};
	\draw[->,>=latex,very thick] (0,0,4.5) -- (0,0,5) node[anchor=west]{$x_3$};
\end{tikzpicture}
\caption{Frequency set $\Ical(U)$ for $\dbf=\begin{psmallmatrix}\exp\\\alg\\\cos\end{psmallmatrix}$ and $U=\{\emptyset,\{1\},\{2\},\{3\},\{1,2\},\{2,3\}\}$ with $\Nbf^{\{1\}}=\begin{psmallmatrix}18\\0\\0\end{psmallmatrix}$, $\Nbf^{\{2\}}=\begin{psmallmatrix}0\\16\\0\end{psmallmatrix}$, $\Nbf^{\{3\}}=\begin{psmallmatrix}0\\0\\10\end{psmallmatrix}$, $\Nbf^{\{1,2\}}=\begin{psmallmatrix}10\\8\\0\end{psmallmatrix}$, and $\Nbf^{\{2,3\}}=\begin{psmallmatrix}0\\6\\8\end{psmallmatrix}$. The frequency set $\tilde{\Ical}_{\Nbf^{\{1\}}}^\dbf$ is shown in green, $\tilde{\Ical}_{\Nbf^{\{2\}}}^\dbf$ in blue, $\tilde{\Ical}_{\Nbf^{\{3\}}}^\dbf$ in red, $\tilde{\Ical}_{\Nbf^{\{1,2\}}}^\dbf$ in cyan, and $\tilde{\Ical}_{\Nbf^{\{2,3\}}}^\dbf$ in magenta.}\label{f:1}
\end{figure}
\begin{expl}
	We show in Figure \ref{f:1} a typical example for such a set $\Ical(U)$, where $\dbf=\begin{psmallmatrix}\exp\\\alg\\\cos\end{psmallmatrix}$ and $U=\{\emptyset,\{1\},\{2\},\{3\},\{1,2\},\{2,3\}\}$. Here we use the bandwidths $\Nbf^{\{1\}}=\begin{psmallmatrix}18\\0\\0\end{psmallmatrix}$, $\Nbf^{\{2\}}=\begin{psmallmatrix}0\\16\\0\end{psmallmatrix}$, $\Nbf^{\{3\}}=\begin{psmallmatrix}0\\0\\10\end{psmallmatrix}$, $\Nbf^{\{1,2\}}=\begin{psmallmatrix}10\\8\\0\end{psmallmatrix}$, and $\Nbf^{\{2,3\}}=\begin{psmallmatrix}0\\6\\8\end{psmallmatrix}$. The parts $\tilde{\Ical}^\dbf_{\Nbf^\ubf}$ of this frequency set $\Ical(U)$ are shown in different colors. The set $\tilde{\Ical}_{\Nbf^{\{1\}}}^\dbf$ is shown in green, $\tilde{\Ical}_{\Nbf^{\{2\}}}^\dbf$ in blue, $\tilde{\Ical}_{\Nbf^{\{3\}}}^\dbf$ in red, $\tilde{\Ical}_{\Nbf^{\{1,2\}}}^\dbf$ in cyan, and $\tilde{\Ical}_{\Nbf^{\{2,3\}}}^\dbf$ in magenta.
\end{expl}

\subsubsection{Approximation}\label{sec:3.2.2}
In this section we assume that the set $U$ and the bandwidths $\Nbf^\ubf$ are known. For a way to choose them we refer to Subsection \ref{sec:3.2.3}. Now, we approximate the truncated function $\Trm_Uf$  with a mixed polynomial $f^\dbf\in \Tcal_{\Ical(U)}(\Bcal^{\dbf})$, where the set of polynomials $\Tcal_{\Ical(U)}(\Bcal^{\dbf})$ is defined in \eqref{eq:30}. As result, we get  $f\approx\Trm_Uf\approx f^\dbf$.\\
The mixed polynomial is completely determined by finitely many mixed coefficients $(c^{\dbf}_\kbf(f^\dbf))_{\kbf\in\Ical(U)}\in\Cbb^{\betr{\Ical(U)}}$. Now it is our goal to find an approximation $\hat{\fbf}^\dbf\approx(c^{\dbf}_\kbf(f^\dbf))_{\kbf\in\Ical(U)}$ to this mixed coefficients. To achieve this, we compute the least squares solution
\begin{align*}
	\normk{f-f^\dbf}_{\Lp{2}(\Dbb^\dbf,\omega^\dbf)}^2=\igl{\Dbb^\dbf}{\betrk{f(\xbf)-f^\dbf(\xbf)}^2\omega^\dbf(\xbf)}{\xbf}.
\end{align*}
We approximate this integral by evaluating the function $\betrk{f(\xbf)-f^\dbf(\xbf)}^2$ at the $M=\betrk{\Xcal}$ given nodes $\Xcal$, where we know the values $\fbf=(f(\xbf))_{\xbf\in\Xcal}$. For this approximation we assume that the nodes in $\Xcal$ are distributed in $\Dbb^\dbf$ with the density $\omega^\dbf$. We get
\begin{align*}
	\normk{f-f^\dbf}_{\Lp{2}(\Dbb^\dbf,\omega^\dbf)}^2&\approx\frac{1}{M}\sum_{\xbf\in\Xcal}\betrk{f(\xbf)-f^\dbf(\xbf)}^2\\
	&=\frac{1}{M}\normk{\fbf-(f^\dbf(\xbf))_{\xbf\in\Xcal}}_2^2\\
	&=\frac{1}{M}\normk{\fbf-\Phibf(\Xcal,\Ical(U))\hat{\fbf}^\dbf}_2^2,
\end{align*}
where $\Phibf(\Xcal,\Ical(U))$ is the matrix $(\phi_\kbf^\dbf(\xbf))_{\xbf\in\Xcal,\;\kbf\in\Ical(U)}\in\Cbb^{M\times\betrk{\Ical(U)}}$. Now we choose $\hat{\fbf}^\dbf$ such that the distance between the function $f$ and the approximation $f^\dbf$ is as small as possible, i.e.
\begin{align}
	\hat{\fbf}^\dbf\coloneqq\argmin_{\hat{\hbf}^\dbf\in\Cbb^{\betr{\Ical(U)}}}\normk{\fbf-\Phibf(\Xcal,\Ical(U)) \hat{\hbf}^\dbf}_2^2\label{eq:5}.
\end{align}
We solve the minimization problem \eqref{eq:5} with the LSQR algorithm \cite{PaSa82}. This algorithm multiplies in every iteration step with the matrix $\Phibf(\Xcal,\Ical(U))$ and its adjoint $\Phibf(\Xcal,\Ical(U))^*$. In Section \ref{sec:4} we develop fast algorithms to multiply with this kind of matrices $\Phibf(\Xcal,\Ical(U))$ and $\Phibf(\Xcal,\Ical(U))^*$. We point out that the number of iterations of the LSQR algorithm depends on the condition number of the underlying matrix. This condition number is in many applications much better than the worst case estimation, see e.g. \cite{BoPo06,BoPoKu07}.\\
We obtain the approximation 
\begin{align*}
	\fkt{f^\dbf}{\Dbb^\dbf}{\Cbb}{\xbf}{\sum_{\kbf\in \Ical(U)}\hat{f}^\dbf_\kbf\phi_\kbf^\dbf(\xbf)}
\end{align*}
for the function $f$. 
Let $\Xcal_\mathrm{test}\subset\Dbb^\dbf,\;\betrk{\Xcal_\mathrm{test}}=M_\mathrm{test}$ be a set with $M_\mathrm{test}\in\Nbb$ nodes, where we evaluate the approximation $f^\dbf$. Then $(f^\dbf(\xbf))_{\xbf\in\Xcal_\mathrm{test}}=\Phibf(\Xcal_\mathrm{test},\Ical(U))\hat{\fbf}^\dbf$ holds. 

\subsubsection{How to choose the truncation set $U$?}\label{sec:3.2.3}
We choose $U$ in two steps, firstly we choose a large superposition set $U_{d_s}$. Then, we calculate the mixed coefficients $\hat{\fbf}^\dbf$ according to this set $U_{d_s}$ for the approximating mixed polynomial $f^\dbf$. In this way we get approximated global sensitivity indices $\rho(\ubf,f^\dbf)$ for every $\ubf\in U_{d_s}$. Using these approximated GSIs we can refine the set $U_{d_s}$ to the final set $U$.\\
Firstly, we choose a superposition dimension  $d_s<d$ and use the corresponding superposition set $U_{d_s}$ which contains only subsets $\ubf$ of the size up to $\betrk{\ubf}\leq d_s\in\Nbb$, see \eqref{eq:31}.
Then we choose appropriate bandwidths $\Nbf^\ubf$. It is good to choose the bandwidths in a way that we have an oversampling, i.e. $\betrk{\Ical(U_{d_s})}<M$ with $\betrk{\Ical(U_{d_s})}=\sum_{\ubf\in U_{d_s}}\betrk{\tilde{\Ical}^\dbf_{\Nbf^\ubf}}$. For $\betrk{\tilde{\Ical}^\dbf_{\Nbf^\ubf}}$ we refer to \eqref{eq:20}.
At this point we refer to Lemma \ref{s:2}, which states that $d_s$ gives an upper bound on the growth rate of $\Ical(U_{d_s})$, i.e. if $d_s=2$ the cardinality of the set $\Ical(U_{d_s})$ grows quadratically with the value of the largest bandwidth.\\
As the next step we calculate the mixed coefficients $\hat{\fbf}^\dbf$ for the approximating mixed polynomial $f^\dbf$ in the way we described it in the previous part. Using these mixed coefficients $\hat{\fbf}^\dbf=(\hat{f}_\kbf^\dbf)_{\kbf\in\Ical(U_{d_s})}$ we calculate the approximated global sensitivity indices $\rho(\ubf,f^\dbf)$ for the mixed polynomial $f^\dbf$ for all $\ubf\in U_{d_s}$ through
\begin{align*}
	\rho(\ubf,f^\dbf)=\frac{\sigma^2(f^\dbf_\ubf)}{\sigma^2(f^\dbf)}=\frac{\displaystyle{\sum_{\substack{\kbf\in\Kbb^\dbf\\\supp\kbf=\ubf}}\betrk{\hat{f}_\kbf^\dbf}^2}}{\displaystyle{\sum_{\kbf\in\Kbb^\dbf\setminus\{\nullbf\}}\betrk{\hat{f}_\kbf^\dbf}^2}}=\frac{\displaystyle{\sum_{\kbf\in\tilde{\Ical}^\dbf_{\Nbf^\ubf}}\betrk{\hat{f}_\kbf^\dbf}^2}}{\displaystyle{\sum_{\kbf\in\Ical(U_{d_s})\setminus\{\nullbf\}}\betrk{\hat{f}_\kbf^\dbf}^2}}.
\end{align*}
Since the mixed polynomial $f^\dbf$ approximates the function $f$, the approximated GSI $\rho(\ubf,f^\dbf)$ should approximate the analytic GSI $\rho(\ubf,f)$.\\
To this end we choose a threshold $\theta>0$ and define the set
\begin{align}
	U_\theta\coloneqq\menk{\ubf\in U_{d_s}}{\rho(\ubf,f^\dbf)>\theta}\label{eq:25}.
\end{align}
With this set $U_\theta$ we do the approximating procedure again. At this stage, we optimize the bandwidths to avoid overfitting and underfitting we do this through cross-validation.

\section{Fast Evaluation of Mixed Polynomials}\label{sec:4}
In this section, we develop a fast algorithm for evaluating sums of the mixed basis functions $\phi_\kbf^\dbf$, i.e. 
\begin{align}
	f^\dbf\coloneqq\sum_{\kbf\in\Ical}\hat{f}^\dbf_\kbf\phi_\kbf^\dbf\label{eq:8}
\end{align}
with arbitrary coefficients $\hat{f}^\dbf_\kbf\in\Cbb$ on a finite index sets $\Ical\subseteq\Kbb^\dbf$ at $M\in\Nbb$ nodes $\Xcal\subset\Dbb^\dbf$, $\betrk{\Xcal}=M$ simultaneously. This evaluation is equivalent to the matrix-vector multiplication of the non-equidistant mixed matrix 
\begin{align}
	\Phibf(\Xcal,\Ical)\coloneqq(\phi_\kbf^\dbf(\xbf))_{\xbf\in\Xcal,\kbf\in\Ical}\label{eq:16}
\end{align}
with the vector $\hat{\fbf}^\dbf=(\hat{f}^\dbf_\kbf)_{\kbf\in\Ical}$, where $\Xcal$ is the set of the nodes at which we are interested to evaluate the mixed polynomial $f^\dbf$. This algorithm is the missing component to do the numerical ANOVA approximation fast with mixed basis functions. In the Subsection \ref{sec:4.1} we consider the index set $\Ical=\Ical_\Nbf^\dbf$, defined in \eqref{eq:1}. Since we will use the thinner index sets $\Ical=\Ical(U)$, defined in \eqref{eq:10}, we will introduce in Subsection \ref{sec:4.2} an algorithm for this case which will rely on the algorithm in Subsection \ref{sec:4.1}.

\subsection{Non-Equidistant Fast Mixed Transform}\label{sec:4.1}
In this subsection, we present a method with a computational cost of only $\Ocal(\betrk{\Ical^\dbf_\Nbf}\log\betrk{\Ical^\dbf_\Nbf}+\betrk{\log\epsilon}^dM)$ for the evaluation of mixed polynomials $f^\dbf$ with the frequency set $\Ical_\Nbf^\dbf$. This is faster than the straightforward matrix vector multiplication with the matrix $\Phibf(\Xcal,\Ical_\Nbf^\dbf)$ which takes $\Ocal(\betrk{\Ical^\dbf_\Nbf}M)$ arithmetical operations.\\
We point out that the mixed polynomial $f^\dbf$ with $\dbf=(\exp,\ldots{},\exp)\eqqcolon\expbf$ is a trigonometric polynomial, which can be evaluated through the non-equidistant fast Fourier transform ($\nfft$) \cite[Chapter 7]{PlPoStTa18} with a computational cost of $\Ocal(\betrk{\Ical^\expbf_\Nbf}\log\betrk{\Ical^\expbf_\Nbf}+\betrk{\log\epsilon}^dM)$, where $\epsilon$ is the required precision, $\betrk{\Ical^\expbf_\Nbf}$ is the cardinality of the frequency set given in \eqref{eq:12}, and $M$ is the number of nodes where we evaluate the mixed polynomial $f^\expbf$. Now, we make use of the $\nfft$ in order to evaluate arbitrary mixed polynomials $f^\dbf$. We use the identity 
\begin{align}
	\sum_{k=0}^{N-1}\hat{f}^{\cos}_k\underbrace{\sqrt{2}^{\;1-\delta_{k,0}}\cos(\pi k x)}_{=\phi_k^{\cos}(x)}=\sum_{k=-N+1}^{N-1}\underbrace{2^{\delta_{k,0}-1}\hat{f}^{\cos}_{\betrk{k}}\sqrt{2}^{1-\delta_{k,0}}}_{\eqqcolon\hat{f}^{\exp}_k}\underbrace{\exp(\pi\im k x)}_{=\phi_k^{\exp}\left(\frac{x}{2}\right)}\label{eq:6m}
\end{align}
to transform cosine polynomials $f^{\cos}$ into trigonometric polynomials $f^{\exp}$. Additionally, we use the identity 
\begin{align}
	\sum_{k=0}^{N-1}\hat{f}^{\alg}_k\underbrace{\sqrt{2}^{\;1-\delta_{k,0}}\cos(k \arccos(2x-1))}_{=\phi_k^{\alg}(x)}=\sum_{k=-N+1}^{N-1}\underbrace{2^{\delta_{k,0}-1}\hat{f}^{\alg}_{\betrk{k}}\sqrt{2}^{1-\delta_{k,0}}}_{\eqqcolon\hat{f}^{\exp}_k}\underbrace{\exp(\im k \arccos(2x-1))}_{=\phi_k^{\exp}\left(\frac{\arccos(2x-1)}{2\pi}\right)}\label{eq:7}.
\end{align}
to transform algebraic polynomials $f^{\alg}$ into trigonometric polynomials $f^{\exp}$.
It follows, that one dimensional polynomials of the form $f^{\cos}$ and $f^{\alg}$ can be evaluated through an $\nfft$. Since our mixed basis functions $\phi_\kbf^\dbf$ have a tensor product structure, we use the identities \eqref{eq:6m} and \eqref{eq:7} in every dimension where the half period cosine basis or the Chebyshev basis is used.
\begin{thm}\label{s:1}
	Let $\hat{\fbf}^\dbf=(\hat{f}^\dbf_{\kbf})_{\kbf\in\Ical_\Nbf^\dbf}\in\Cbb^{\betr{\Ical_\Nbf^\dbf}}$ be a coefficient vector for a mixed polynomial $f^\dbf$ defined in \eqref{eq:8} and an arbitrary $\dbf\in\{\exp,\cos,\alg\}^d$ and $d\in\Nbb$. We define the coefficient vector $\hat{\fbf}^{\expbf}=(\hat{f}^{\expbf}_\kbf)_{\kbf\in\Ical_\Nbf^{\expbf}}\in\Cbb^{\betrk{\Ical_\Nbf^{\expbf}}}$ through
	\begin{align}
		\hat{f}^{\expbf}_\kbf\coloneqq \hat{f}_{\sbf(\kbf)}^\dbf\prod_{j=1}^{d}
		\begin{cases}
			1,&d_j=\exp\text{ or }k_j=0\\
			0,&d_j\neq\exp\text{ and }k_j=-\frac{N_j}{2}\\
			\frac{\sqrt{2}}{2},&\text{else}\label{eq:13}
		\end{cases}
	\end{align}
    for all $\kbf\in\Ical_\Nbf^{\expbf}$, where $\sbf$ is the index transformation
	\begin{align}
		\fkt{\sbf}{\Ical_\Nbf^\expbf}{\Ical_\Nbf^\dbf}{\kbf}{\left(
		\begin{cases}
			k_j,& d_j=\exp\\
			\betrk{k_j},&d_j\neq\exp\text{ and }k_j\neq-\frac{N_j}{2}\\
			0,&d_j\neq\exp\text{ and }k_j=-\frac{N_j}{2}
		\end{cases}
		\right)_{j=1}^d}\label{eq:15}.
	\end{align}
	Furthermore, we define the function $\fkt{\tbf}{\Dbb^\dbf}{\Dbb^{\expbf}}{\xbf}{(t_j(x_j))_{j=1}^d}$ with
	\begin{align*}
		\fkt{t_j}{\Dbb^{d_j}}{\Dbb^{\exp}}{x}{
			\begin{cases}
				x,&d_j=\exp\\
				\frac{x}{2},&d_j=\cos\\
				\frac{\arccos(2x-1)}{2\pi},&d_j=\alg
			\end{cases}}.
	\end{align*}
	Then the identity $f^{\dbf}=f^{\expbf}\circ\tbf$ holds.
\end{thm}
\begin{proof}
	We point out that 
	\begin{align}
		\phi_{k_j}^{\exp}(t_j(x_j))+\phi_{-k_j}^{\exp}(t_j(x_j))=(\exp(2\pi\im k_jt_j(x_j))+\exp(-2\pi\im k_jt_j(x_j)))=2\cos(2\pi k_kt_j(x_j))\label{eq:22}
	\end{align}
	holds. Furthermore 
	\begin{align}
		\sqrt{2}\cos(2\pi k_jt_j(x_j))=\phi_{k_j}^{d_j}(x_j)\text{ for }d_j=\cos,\alg\label{eq:23}
	\end{align}
	and $\phi_{0}^{\exp}=\phi_{0}^{\cos}=\phi_{0}^{\alg}$ holds by definition.
	We have a look at the trigonometric polynomial $f^{\expbf}\circ\tbf$ at a node $\xbf\in\Dbb^{\expbf}$ and obtain
	\begin{align*}
		f^{\expbf}(\tbf(\xbf))&=\sum_{\kbf\in\Ical^{\expbf}_\Nbf}\hat{f}^{\expbf}_\kbf\phi^{\expbf}_\kbf(\tbf(\xbf))\\
		&=\sum_{\kbf\in\Ical^{\expbf}_\Nbf}\hat{f}_{\sbf(\kbf)}^\dbf\prod_{j=1}^{d}
		\begin{cases}
			1\phi_{k_j}^{\exp}(t_j(x_j)),&d_j=\exp\text{ or }k_j=0\\
			0\phi_{k_j}^{\exp}(t_j(x_j)),&d_j\neq\exp\text{ and }k_j=-\frac{N_j}{2}\\
			\frac{\sqrt{2}}{2}\phi_{k_j}^{\exp}(t_j(x_j)),&\text{else}
		\end{cases}\\
		&=\sum_{\kbf\in\Ical^{\expbf}_\Nbf}\hat{f}_{\sbf(\kbf)}^\dbf\prod_{j=1}^{d}
		\begin{cases}
			\phi_{k_j}^{\exp}(x_j),&d_j=\exp\\
			\phi_{0}^{\exp}(t_j(x_j)),&d_j\neq\exp\text{ and }k_j=0\\
			0,&d_j\neq\exp\text{ and }k_j=-\frac{N_j}{2}\\
			\frac{\sqrt{2}}{2}\phi_{k_j}^{\exp}(t_j(x_j)),&\text{else}
		\end{cases}\\
		&=\sum_{\kbf\in\Ical^{\dbf}_\Nbf}\hat{f}_{\kbf}^\dbf\prod_{j=1}^{d}
		\begin{cases}
			\phi_{k_j}^{\exp}(x_j),&d_j=\exp\\
			\phi_{0}^{\exp}(t_j(x_j)),&d_j\neq\exp\text{ and }k_j=0\\
			\frac{\sqrt{2}}{2}\phi_{k_j}^{\exp}(t_j(x_j))+\frac{\sqrt{2}}{2}\phi_{-k_j}^{\exp}(t_j(x_j)),&\text{else}
		\end{cases}\\
		&\underset{\eqref{eq:22}}{=}\sum_{\kbf\in\Ical^{\dbf}_\Nbf}\hat{f}_{\kbf}^\dbf\prod_{j=1}^{d}
		\begin{cases}
			\phi_{k_j}^{\exp}(x_j),&d_j=\exp\\
			\phi_{0}^{\exp}(t_j(x_j)),&d_j\neq\exp\text{ and }k_j=0\\
			\sqrt{2}\cos(2\pi k_jt_j(x_j)),&\text{else}
		\end{cases}\\
		&\underset{\eqref{eq:23}}{=}\sum_{\kbf\in\Ical^{\dbf}_\Nbf}\hat{f}_{\kbf}^\dbf\prod_{j=1}^{d}\phi_{k_j}^{d_j}(x_j)\\
		&=f^\dbf(\xbf).
	\end{align*}
\end{proof}
\begin{rem}
The Theorem \ref{s:1} provides us a decomposition of the non-equidistant mixed matrix $\Phibf(\Xcal,\Ical_\Nbf^\dbf)$. For this purpose we define the diagonal matrix 
\begin{align*}
	\Dbf\coloneqq\diag\left(\prod_{j=1}^{d}
	\begin{cases}
		1,&d_j=\exp\text{ or }k_j=0\\
		0,&d_j\neq\exp\text{ and }k_j=-\frac{N_j}{2}\\
		\frac{\sqrt{2}}{2},&\text{else}
	\end{cases}\right)_{\kbf\in\Ical_\Nbf^\dbf},
\end{align*}
the canonical map $\Pibf\coloneqq(\delta_{\kbf,\lbf})_{\kbf\in\Ical_\Nbf^{\dbf},\lbf\in\Ical_\Nbf^\expbf}$, the projection $\Pbf\coloneqq(\delta_{\kbf,\sbf(\kbf)})$, and the non-equidistant Fourier matrix $\Abf=(\exp(2\pi\im\skprk{\kbf}{\tbf(\xbf)}{}))_{\xbf\in\Xcal,\kbf\in\Ical_\Nbf^\expbf}$. Then the matrix transformation 
\begin{align}\label{eq:3}
    \Phibf(\Xcal,\Ical_\Nbf^\dbf)=\Abf\Pibf^\top\Pbf^\top\Dbf
\end{align}
follows directly.
\end{rem}
We summarize the procedure for the efficient evaluation of the mixed polynomials $f^\dbf$ at $M$ arbitrary nodes as non-equidistant fast mixed transform ($\nffct$) in Algorithm \ref{alg:1}.\\
\begin{algorithm}[H]
    \KwIn{Vector $\dbf\in\{\exp,\cos,\alg\}^d$, bandwidths $\Nbf\in(2\Nbb)^{d}$, coefficients $\hat{f}^\dbf_\kbf\in\Cbb$ for $\kbf\in\Ical_\Nbf^\dbf$, nodes $\Xcal\subset\Dbb^\dbf$, $\betrk{\Xcal}=M$}
    Define the coefficients $\hat{f}^{\expbf}_\kbf$ given in \eqref{eq:13} for all $\kbf\in\Ical_\Nbf^{\expbf}$.\\
    Compute
    \begin{align*}
        s(\xbf)=\sum_{\kbf\in\Ical_\Nbf^\dbf}\hat{f}^{\expbf}_\kbf\exp\left({2\pi\im\skp{\kbf}{\tilde{\xbf}}}\right)
    \end{align*}
    at the nodes $\tilde{\xbf}\in\menk{\tbf(\xbf)}{\xbf\in\Xcal}$ with $\fktk{\tbf}{\Dbb^\dbf}{\Dbb^{\expbf}}$ defined in Theorem \ref{s:1} using a $d$-variate $\nfft$\\
    \KwOut{$s(\xbf)= f^\dbf(\xbf)$ for $\xbf\in\Xcal$}
    \KwCc{$\Ocal(\betrk{\Ical_\Nbf^\dbf}\log\betrk{\Ical_\Nbf^\dbf}+\betrk{\log\epsilon}^dM)$}
    \caption{$\nffct$ for the fast evaluation of mixed polynomials $f^\dbf$ for frequency sets $\Ical_\Nbf^\dbf$ defined in \eqref{eq:1}.}\label{alg:1}
\end{algorithm}
In addition we evaluate the sum 
\begin{align}
    h(\kbf)=\sum_{\xbf\in\Xcal} g_\xbf\phi_\kbf^\dbf(\xbf),\;g_\xbf\in\Cbb\label{eq:2}
\end{align}
for all $\kbf\in\Ical^\dbf_\Nbf$. This is equivalent to the matrix vector product of the transposed non-equidistant mixed matrix $\Phibf(\Xcal,\Ical_\Nbf^\dbf)^\top$ with the vector $\gbf=(g_\xbf)_{\xbf\in\Xcal}$. We use the factorization \eqref{eq:3} of the non-equidistant mixed matrix $\Phibf(\Xcal,\Ical_\Nbf^\dbf)$ and get directly the Algorithm \ref{alg:2}, which provides a method for the fast evaluation of the sum \eqref{eq:2}.\\
\begin{algorithm}[h]
    \KwIn{Vector $\dbf\in\{\exp,\cos,\alg\}^d$, bandwidths $\Nbf\in(2\Nbb)^{d}$, nodes $\Xcal\subset\Dbb^\dbf$, $\betrk{\Xcal}=M$, coefficients $h_\xbf\in\Cbb$ for $\xbf\in\Xcal$}
    Compute
    \begin{align*}
        \tilde{h}^{\exp}(\kbf)=\sum_{\xbf\in\Xcal} h_\xbf\exp\left({2\pi\im\skp{\kbf}{\tilde{\xbf}}}\right)
    \end{align*}
    with $\kbf\in\Ical_\Nbf^\dbf$ at the nodes $\tilde{\xbf}\in\menk{\tbf(\xbf)}{\xbf\in\Xcal}$ with $\fktk{\tbf}{\Dbb^\dbf}{\Dbb^\expbf}$ defined in Theorem \ref{s:1} using a $d$-variate $\nfftt$\\
    Compute
    \begin{align*}
    \tilde{h}(\kbf)=\sum_{\substack{\lbf\in\Ical^\expbf_\Nbf\\\sbf(\lbf)=\kbf}}\tilde{h}^{\exp}(\lbf)\prod_{j=1}^{d}
	\begin{cases}
		1,&d_j=\exp\text{ or }k_j=0\\
		0,&d_j\neq\exp\text{ and }k_j=-\frac{N_j}{2}\\
		\frac{\sqrt{2}}{2},&\text{else}
	\end{cases}
    \end{align*}
	with $\sbf$ defined in \eqref{eq:15}.\\
    \KwOut{$\tilde{h}(\kbf)= h(\kbf)$, see \eqref{eq:2}, for $\kbf\in\Ical_\Nbf^\dbf$}
    \KwCc{$\Ocal(\betrk{\Ical_\Nbf^\dbf}\log\betrk{\Ical_\Nbf^\dbf}+\betrk{\log\epsilon}^dM)$}
    \caption{$\nffctt$ for the fast evaluation of sums of the form \eqref{eq:2} for frequency sets $\Ical_\Nbf^\dbf$ defined in \eqref{eq:1}.}\label{alg:2}
\end{algorithm}
\begin{rem}
	\begin{itemize}\itemsep0em
		\item We note that one can extend the non-equidistant fast mixed transformations to other orthogonal polynomials using \cite{po01}. The algorithm known as discrete polynomial transform provides a fast basis exchange for arbitrary orthogonal polynomials with satisfying a three-term recurrence into the Chebyshev basis. These Chebyshev polynomials can then be evaluated by the Algorithms \ref{alg:1} and \ref{alg:2}.
		\item One can also use other transformations, such as the transformation of the unit interval [0,1] into the real numbers $\Rbb$ from \cite{PoSc22}. This allows us to handle normally distributed nodes.
	\end{itemize}
It should be noted that these transformations can be performed in each dimension separately due to the tensor product structure of the basis and the flexibility of the mixed basis.
\end{rem}

\subsection{Grouped Mixed Transformations}\label{sec:4.2}
In this section we derive a fast algorithm for the evaluation of mixed polynomials $f^\dbf=\sum_{\kbf\in \Ical(U)}\hat{f}^\dbf_\kbf\phi_\kbf^\dbf$ where $\Ical(U)$ is a frequency set defined in \eqref{eq:10}. The theory of grouped transformations for trigonometric polynomials is well established in \cite{BaPoSchmi21}. In the following section, we summarise this idea of grouped transformations and apply it to the case of the mixed polynomials.
We denote $\xbf_\ubf\coloneqq\Pibf_\ubf\xbf$, where $\Pibf_\ubf$ is the canonical map $\Pibf_\ubf$ onto the dimensions contained in $\ubf$.\\
The evaluation of this sum at the nodes $\xbf\in\Xcal$ is equivalent to calculate the matrix vector product $\fbf^\dbf=\Phibf(\Xcal,\Ical(U))\hat{\fbf}^\dbf\in\Cbb^M$ with the matrix $\Phibf(\Xcal,\Ical(U))$, defined in \eqref{eq:16} and the vector $\hat{\fbf}^\dbf\in\Cbb^{\betr{\Ical(U)}}$. We follow the steps in \cite{BaPoSchmi21} and get
\begin{align*}
	\sum_{\kbf\in \Ical(U)}\hat{f}^\dbf_\kbf\phi_\kbf^\dbf(\xbf)=\sum_{\ubf\in U}\sum_{\kbf\in\tilde{\Ical}_{\Nbf^\ubf}^\dbf}\hat{f}^\dbf_\kbf\phi_\kbf^\dbf(\xbf)
\end{align*}
through the structure of the frequency set $\Ical(U)$. In other words, the matrix $\Phibf(\Xcal,\Ical(U))$ is a block matrix with horizontally arranged blocks $\Phibf(\Xcal,\tilde{\Ical}_{\Nbf^\ubf}^\dbf)$, $\ubf\in U$, i.e. $\Phibf(\Xcal,\Ical(U))=(\Phibf(\Xcal,\tilde{\Ical}_{\Nbf^\ubf}^\dbf)^\top)_{\ubf\in U}^\top$. Thus, we divide the task. For every $\ubf\in U$ we multiply the vector $\hat{\fbf}^{\dbf,\ubf}\coloneqq(\hat{f}^\dbf_\kbf)_{\kbf\in\tilde{\Ical}_{\Nbf^\ubf}^\dbf}$ with the block  $\Phibf(\Xcal,\tilde{\Ical}_{\Nbf^\ubf}^\dbf)$. We get for these blocks 
\begin{align*}
	\Phibf(\Xcal,\tilde{\Ical}_{\Nbf^\ubf}^\dbf)=(\phi_\kbf^{\dbf_\ubf}(\xbf_\ubf))_{\xbf\in\Xcal,\;\kbf\in \tilde{\Ical}_{\Nbf^\ubf}^{\dbf_\ubf}}=\Phibf(\men{\xbf_\ubf}{\xbf\in\Xcal},\tilde{\Ical}_{\Nbf^\ubf}^{\dbf_\ubf}).
\end{align*}
We define the vector $\hat{\gbf}^{\dbf_\ubf}=(\hat{g}^{\dbf_\ubf}_\kbf)_{\kbf\in\Ical_{\Nbf_\ubf^\ubf}^{\dbf_\ubf}}\in\Cbb^{\betrk{\Ical_{\Nbf_\ubf^\ubf}^{\dbf_\ubf}}}$, where now $\Ical_{\Nbf_\ubf^\ubf}^{\dbf_\ubf}$ is a frequency set which we can use for an $\nffct$. At this point we would like to clarify that $\Nbf^\ubf_\ubf$ are exactly the non-zero entries of $\Nbf^\ubf$. So $\Ical_{\Nbf_\ubf^\ubf}^{\dbf_\ubf}$ is a set of frequencies $\kbf$ of lower dimension as the frequencies $\tilde{\kbf}$ in $\tilde{\Ical}_{\Nbf^\ubf}^\dbf$. Specifically, $\kbf$ are the projections of $\tilde{\kbf}$ onto the dimensions contained in $\ubf$. We set each component of $\hat{\gbf}^{\dbf_\ubf}$ which is not contained in the set $\tilde{\Ical}_{\Nbf^\ubf}^{\dbf_\ubf}$ to zero, e.g.
\begin{align*}
	\hat{g}_\kbf=
	\begin{cases}
		\hat{f}_\kbf,&\betr{\supp\kbf}=\betr{\ubf}\\
		0,&\text{else}
	\end{cases},\kbf\in\Ical_{\Nbf_\ubf^\ubf}^{\dbf_\ubf}.
\end{align*}
Then we obtain
\begin{align*}
	\sum_{\kbf\in\Ical_{\Nbf_\ubf^\ubf}^{\dbf_\ubf}}\hat{g}^{\dbf_\ubf}_\kbf\phi_\kbf^{\dbf_\ubf}(\xbf)&=\sum_{\kbf\in\tilde{\Ical}_{\Nbf^\ubf}^{\dbf_\ubf}}\hat{f}^\dbf_\kbf\phi_\kbf^{\dbf_\ubf}(\xbf)+\sum_{\kbf\in\Ical_{\Nbf_\ubf^\ubf}^{\dbf_\ubf}\setminus\tilde{\Ical}_{\Nbf^\ubf}^{\dbf_\ubf}}0\cdot\phi_\kbf^{\dbf_\ubf}(\xbf)=\sum_{\kbf\in\tilde{\Ical}_{\Nbf^\ubf}^{\dbf_\ubf}}\hat{f}^\dbf_\kbf\phi_\kbf^{\dbf_\ubf}(\xbf)
\end{align*}
or in matrix vector form $\Phibf(\Xcal,\tilde{\Ical}_{\Nbf^\ubf}^{\dbf_\ubf})\hat{\fbf}^\dbf=\Phibf(\Xcal,\Ical_{\Nbf_\ubf^\ubf}^{\dbf_\ubf})\hat{\gbf}^{\dbf_\ubf}$. Which is equivalent to the matrix decomposition $\Phibf(\Xcal,\tilde{\Ical}_{\Nbf^\ubf}^{\dbf_\ubf})=\Phibf(\Xcal,\Ical_{\Nbf_\ubf^\ubf}^{\dbf_\ubf})\tilde{\Pibf}^\top$ where $\tilde{\Pibf}=(\delta_{\lbf,\kbf})_{\lbf\in\tilde{\Ical}_{\Nbf^\ubf}^{\dbf_\ubf},\kbf\in\Ical_{\Nbf_\ubf^\ubf}^{\dbf_\ubf}}\in\Rbb^{\betrk{\tilde{\Ical}_{\Nbf^\ubf}^{\dbf_\ubf}}\times\betrk{\Ical_{\Nbf_\ubf^\ubf}^{\dbf_\ubf}}}$ is a canonical map. To sum this up, we multiply the matrix $\Phibf(\Xcal,\Ical(U))$ with the vector $\hat{\fbf}^\dbf\in\Cbb^{\betr{\Ical(U)}}$ through calculating
\begin{align*}
	\Phibf(\Xcal,\Ical(U))\hat{\fbf}^\dbf&=\sum_{\ubf\in U}\Phibf(\Xcal,\Ical_{\Nbf^\ubf}^\dbf)\hat{\fbf}^{\dbf,\ubf}\\
	&=\sum_{\ubf\in U}\Phibf(\men{\xbf_\ubf}{\xbf\in\Xcal},\tilde{\Ical}_{\Nbf^\ubf}^{\dbf_\ubf})\hat{\fbf}^{\dbf,\ubf}\\
	&=\sum_{\ubf\in U}\Phibf(\men{\xbf_\ubf}{\xbf\in\Xcal},\Ical_{\Nbf_\ubf^\ubf}^{\dbf_\ubf})\tilde{\Pibf}^\top\hat{\fbf}^{\dbf,\ubf}.
\end{align*}
We calculate the last sum with $\betrk{U}$ many $\nffct$. This leads us to a computational cost of \\$\Ocal(\sum_{\ubf\in U}(\betrk{\Ical_{\Nbf_\ubf^\ubf}^{\dbf_\ubf}}\log\betrk{\Ical_{\Nbf_\ubf^\ubf}^{\dbf_\ubf}}+m_\nfft^{\betr{\ubf}}M))$. Additionally, it can be easily parallelized, because every summand can be computed independently. We summarize the this procedure in Algorithm \ref{alg:3}\\
\begin{algorithm}[H]
	\KwIn{Vector $\dbf\in\{\exp,\cos,\alg\}^d$, truncation set $U$, bandwidths $\Nbf^\ubf\in(2\Nbb)^\ubf$ for $\ubf\in U$, coefficients $\hat{f}_\kbf\in\Cbb$ for all $\kbf\in\Ical(U)$, nodes $\Xcal\subset\Dbb^\dbf$, $\betrk{\Xcal}=M$}
	$\fbf\leftarrow\nullbf$\\
	\ForEach{$\ubf\in U$\tcp*[f]{This loop can be parallelized}\\}{
		$\tilde{\Xcal}\leftarrow\men{\xbf_\ubf}{\xbf\in\Xcal}$\\
		$\hat{g}_\kbf\leftarrow
		\begin{cases}
			\hat{f}_\kbf,&\betr{\supp\kbf}=\betr{\ubf}\\
			0,&\text{else}
		\end{cases},\;\kbf\in\Ical_{\Nbf_\ubf^\ubf}^{\dbf_\ubf}$\\
		Compute $\fbf\leftarrow\fbf+\Phibf(\tilde{\Xcal},\Ical_{\Nbf_\ubf^\ubf}^{\dbf_\ubf})\gbf$ using a $\betrk{\ubf}$-variate $\nffct$\\
	}
	\KwOut{$\fbf= \Phibf(\Xcal,\Ical(U))\hat{\fbf}$}
	\KwCc{$\Ocal\left(\sum\limits_{\ubf\in U}\left(\betrk{\Ical_{\Nbf_\ubf^\ubf}^{\dbf_\ubf}}\log\betrk{\Ical_{\Nbf_\ubf^\ubf}^{\dbf_\ubf}}+m_\nfft^{\betr{\ubf}}M\right)\right)$}
	\caption{Grouped transform for the fast evaluation of mixed polynomials $f^\dbf$ with a frequency set $\Ical(U)$, see \eqref{eq:10}.}\label{alg:3}
\end{algorithm}
Furthermore, the identity 
\begin{align}
	\Phibf(\Xcal,\Ical(U))=\left(\tilde{\Pibf}\Phibf(\men{\xbf_\ubf}{\xbf\in\Xcal},\Ical_{\Nbf_\ubf^\ubf}^{\dbf_\ubf})^\top\right)^\top_{\ubf\in U}\label{eq:11}
\end{align}
holds. This \eqref{eq:11} leads us to an algorithm for multiplying with the adjoint matrix $\Phibf(\Xcal,\Ical(U))^\ast$, because
\begin{align*}
	\Phibf(\Xcal,\Ical(U))^\ast=\left(\tilde{\Pibf}\Phibf(\men{\xbf_\ubf}{\xbf\in\Xcal},\Ical_{\Nbf_\ubf^\ubf}^{\dbf_\ubf})^\ast\right)_{\ubf\in U}
\end{align*}
holds. Thus, we have a fast algorithm for the evaluation of the sum
\begin{align}
	k(\kbf)=\sum_{\xbf\in\Xcal}h_\xbf\bas{\kbf}{m}{n}(\xbf)\label{eq:18}
\end{align}
for coefficients $\hbf=(h_\xbf)_{\xbf\in\Xcal}\in\Cbb^{M}$ at the nodes $\kbf\in\Ical(U)$, which we summarize as Algorithm \ref{alg:4}.\\
\begin{algorithm}[h]
	\KwIn{Vector $\dbf\in\{\exp,\cos,\alg\}^d$, truncation set $U$, bandwidths $\Nbf^\ubf\in(2\Nbb)^\ubf$ for $\ubf\in U$, nodes $\Xcal\subset\Dbb^\dbf$, $\betrk{\Xcal}=M$, coefficients $h_\xbf\in\Cbb$ for $\xbf\in\Xcal$}
	\ForEach{$\ubf\in U$\tcp*[f]{This loop can be parallelized}\\}{
		$\tilde{\Xcal}\leftarrow\men{\xbf_\ubf}{\xbf\in\Xcal}$\\
		Compute $\gbf^\ubf\leftarrow\Phibf(\tilde{\Xcal},\Ical_{\Nbf_\ubf^\ubf}^{\dbf_\ubf})^*(h_\xbf)_{\xbf\in\Xcal}$ using a $\betrk{\ubf}$-variate $\nffctt$\\
		$\fbf^\ubf\leftarrow(g^\ubf_j)_{j\in\Ical_{\Nbf_\ubf^\ubf}^{\dbf_\ubf}}$\\
	}
	\KwOut{$(\fbf^\ubf)_{\ubf\in U}= \Phibf(\Xcal,\Ical(U))^*(h_\xbf)_{\xbf\in\Xcal}$}
	\KwCc{$\Ocal\left(\sum\limits_{\ubf\in U}\left(\betrk{\Ical_{\Nbf_\ubf^\ubf}^{\dbf_\ubf}}\log\betrk{\Ical_{\Nbf_\ubf^\ubf}^{\dbf_\ubf}}+m_\nfftt^{\betr{\ubf}}M\right)\right)$}
	\caption{Adjoint grouped transformation for the fast evaluation of the sum \eqref{eq:18} for frequency sets $\Ical(U)$ defined in \eqref{eq:10}.}\label{alg:4}
\end{algorithm}

\section{Numerical Experiments}\label{sec:5}
In this subsection, we test the ANOVA approximation with the mixed bases on synthetic and real data. In Subsection \ref{sec:5.1}, we show how the approximation procedure works and how we determine the bandwidths in this case. In Subsection \ref{sec:5.2} we compare the ANOVA approximation with the mixed bases to the ANOVA approximation with a fully periodic and a fully non-periodic basis, respectively. Furthermore, we compare analytic global sensitivity indices in Appendix \ref{c:1} to approximated ones. Furthermore, we investigate here the empirical convergence behaviour of the different approximation methods for this function. In Subsection \ref{sec:5.3} we apply the ANOVA approximation with mixed bases on a dataset from a real application. We show the improvement of the error in comparison to the ANOVA approximation with the cosine basis.\\
We have extended the ANOVAapprox framework \cite{gitNFFT3jl,gitGTjl,gitANOVAjl} with the algorithms listed in Section \ref{sec:4} and run all the following tests in this framework.\\
To determine the quality of the ANOVA approximation $\tilde{f}$ for a function $f$, we consider the mean squared error (MSE),
\begin{align*}
	\mop{MSE}(f,\tilde{f},\Xcal_\textrm{test})\coloneqq\frac{1}{\betr{\Xcal_\textrm{test}}}\sum_{\xbf\in\Xcal_\textrm{test}}\betr{f(\xbf)-\tilde{f}(\xbf)}^2,
\end{align*}
at the nodes $\Xcal_\textrm{test}$. We compute the MSE multiple times for randomly chosen training nodes $\Xcal$ and test nodes $\Xcal_\textrm{test}$. We denote how many times we compute the MSE with $N_{\textrm{MSE}}\in\Nbb$. We then average over the MSEs.
\subsection{ANOVA approximation with a mixed basis}\label{sec:5.1}
In this subsection we approximate a function using the ANOVA approximation with the mixed basis. A special focus lays in the question, how we determine the truncation set $U$ and the according bandwidths numerically. The function we are approximating in this section is
\begin{align*}
	\fkt{f_1}{[0,1]^4}{\Cbb}{\xbf}{\exp(\sin(2\pi x_1)x_2) + \cos(\pi x_3)x_4^2 + \frac{1}{10}\sin^2(2\pi x_1) + 5\sqrt{x_2x_4+1}}.
\end{align*}
This function $f_1$ is smoothly periodizable in the first dimension, i.e. $\fkt{f^{\textrm{per}}_1}{\Tbb\times[0,1]^3}{\Cbb}{\xbf}{f_1(\xbf)}$ is infinitely differentiable. Furthermore, the function acts in the third dimension only as a cosine function. This leads us to use the mixed basis $\phi_\kbf^{\dbf_1}$ with $\dbf_1\coloneqq(\exp,\alg,\cos,\alg)^\top$ for $\kbf\in\Kbb^{\dbf_1}$ for approximating the function $f_1$.\\
For this approximation we restrict us to only 1000 nodes $\Xcal$ in $\Dbb^{\dbf_1}$ distributed with the density 
\begin{align*}
	\omega^{\dbf_1}(\xbf)=\frac{1}{\pi^2\sqrt{x_2-x_2^2}\sqrt{x_4-x_4^2}}.
\end{align*}
Furthermore, we are given another 10000 nodes $\Xcal_\textrm{test}$ in $\Dbb^{\dbf_1}$ distributed with the density $\omega^{\dbf_1}$ for evaluating the mean squared error.\\
We follow the steps from Section \ref{sec:3.2.3}. The function has only one dimensional and two-dimensional interactions between variables. Thus, we set $d_s=2$ and consider the superposition set \newline$U_2=\{\{1\},\{2\},\{3\},\{4\},\{1,2\},\{1,3\},\{1,4\},\{2,3\},\{2,4\},\{3,4\}\}$ from \eqref{eq:21}. We choose one bandwidth parameter $N_1\in(2\Nbb)$ for the one-dimensional frequency sets, e.g. every non-zero entry of $\Nbf^{\{1\}}$, $\Nbf^{\{2\}}$, $\Nbf^{\{3\}}$ and $\Nbf^{\{4\}}$ is set to $N_1$. Furthermore, we choose another bandwidth parameter $N_2\in(2\Nbb)$ for the two-dimensional frequency sets, e.g. every non-zero entry of $\Nbf^{\{1,2\}}$, $\Nbf^{\{1,3\}}$, $\Nbf^{\{1,4\}}$, $\Nbf^{\{2,3\}}$, $\Nbf^{\{2,4\}}$ and $\Nbf^{\{3,4\}}$ is set to $N_2$. In short form we write this as $\Nbf^\ubf=(\betrk{\{i\}\cap\ubf}N_{\betrk{\ubf}})_{i=1}^4$ for $\ubf\in U_2$. We call the approximation of the function $f_1$ using the 1000 nodes $\Xcal$ and the bandwidth parameters $N_1$ and $N_2$ $\tilde{f}_1^{N_1,\;N_2}$. We determine the optimal bandwidth parameters $N_1$ and $N_2$ numerically by minimizing the mean squared error $\mop{MSE}(f_1,\tilde{f}_1^{N_1,\;N_2},\Xcal_\textrm{test})$, i.e.
\begin{align}
	(N_i)_{i=1}^2=\argmin_{(N_i)_{i=1}^2\in(2\Nbb)^2}\mop{MSE}(f_1,\tilde{f}_1^{N_1,N_2},\Xcal_\textrm{test})\label{eq:29}.
\end{align}
In other words, we use cross validation to determine the bandwidth parameters $N_1$ and $N_2$.

\begin{figure}
	\begin{minipage}[b]{.4\linewidth}
		\begin{tikzpicture}
		\begin{axis}[
            view={0}{90},   
            xlabel=$N_2$,
            ylabel=$N_1$,
            colorbar,
			width=7cm,
			height=6cm,
            enlargelimits=false,
            axis on top,
			colorbar style={ytick={-7,-5,...,-1},yticklabel={$10^{\pgfmathparse{\tick}\pgfmathprintnumber\pgfmathresult}$}},
			y label style={at={(axis description cs:0.03,0.5)},anchor=north},
        ]
            \addplot [matrix plot*,point meta=explicit] file {data3.dat};
        \end{axis}
	  \end{tikzpicture}
	  \caption{Mean squared errors for $U_2$ and bandwidths $\Nbf^\ubf=(\betrk{\{i\}\cap\ubf}N_{\betrk{\ubf}})_{i=1}^4$ and $N_1$ and $N_2$ for $f_1$.}\label{dg:1}
	\end{minipage}
	\hspace{.1\linewidth}
	\begin{minipage}[b]{.5\linewidth}
		\begin{tikzpicture}
	\begin{semilogyaxis}[name=boundary,ymin=0,ymax=1,axis y line=left,axis x line*=bottom,width=7cm,height=6cm,symbolic x coords/.expand once={$\{1\}$,$\{2\}$,$\{3\}$,$\{4\}$,$\{1;2\}$,$\{1;3\}$,$\{1;4\}$,$\{2;3\}$,$\{2;4\}$,$\{3;4\}$},xtick=data,x tick label style={rotate=90,anchor=east},grid=major,xticklabel style = {yshift=-7},]
	\addplot+[ycomb,mark=triangle*,shift={(50,0)}, color=violet, solid,mark options={fill=violet}] table [x=a, y=b, col sep=space] {data4.dat};\label{pl:5.1}

	\end{semilogyaxis}
	\end{tikzpicture}
	\caption{Approximated global sensitivity indices for $U_2$ and bandwidths $\Nbf^\ubf=(\betrk{\{i\}\cap\ubf}N_{\betrk{\ubf}})_{i=1}^4$ and $N_1=12$ and $N_2=10$ for $f_1$.}\label{dg:2}
\end{minipage}
\end{figure}

We see in Figure \ref{dg:1} the MSE for some choices of $N_1$ and $N_2$. We vary the parameter $N_1$ from 2 to 50 and the parameter $N_2$ from 2 to 12. The colour of each block corresponds to the MSE of the approximation with the corresponding parameter set. We obtain values for the MSE in the range $[10^{-7.7133},10^{-0.4968}]$. The minimum is obtained with the parameter set $N_1=12$ and $N_2=10$. In Figure \ref{dg:2} the approximated GSIs for the approximation with these parameters are shown. We show these GSIs on a logarithmic scale because they are of quite different magnitudes. To this end we choose the threshold $\theta=10^{-2}$ and find through \eqref{eq:25} the truncation set $U_\theta=\{\{1\},\{2\},\{3\},\{4\},\{1,2\},\{2,4\},\{3,4\}\}$. Next, we find better bandwidths $\Nbf^\ubf$ for $U_\theta$. To do this we introduce a new set of bandwidth parameters $N_{\ubf}\in(2\Nbb)$ for $\ubf\in U_\theta$, i.e. one parameter for every bandwidth. We get the bandwidths $\Nbf^\ubf$ by setting every non-zero entry to $N_{\ubf}$, i.e. $\Nbf^\ubf=(\betrk{\{i\}\cap\ubf}N_{\ubf})_{i=1}^4$. We optimise these bandwidth parameters one by one, starting with the parameters corresponding to the two-dimensional bandwidths. We do this through increasing the parameter firstly bigger until the MSE gets bigger. If the MSE gets bigger in the first step, we decrease the parameter until the MSE gets bigger. Then we use the parameter which has generated the minimal MSE. As a starting point we use the bandwidths $\Nbf^\ubf=(\betr{\{i\}\cap \ubf}N_{\betr{\ubf}})_{i=1}^4$ generated by the optimal parameters $N_1=12$ and $N_2=10$ of the previous approximation step. In Table \ref{tb:1} we show the parameters we tried to find the optimal ones.

\begin{table}[!h]
	\small
	\centering
	\begin{tabular}[h]{|R{2em}||R{2.5em}|R{2.5em}|R{2.5em}|R{2.5em}|R{2.5em}|R{2.5em}|R{2.5em}||R{6.5em}|}
	\hline
	Step&$N_{\{1\}}$&$N_{\{2\}}$&$N_{\{3\}}$&$N_{\{4\}}$&$N_{\{1,2\}}$&$N_{\{2,4\}}$&$N_{\{3,4\}}$&MSE\\
	\hline\hline
	1&12&12&12&12&10&10&10&$1.31369\cdot10^{-8}$\\\hline
	2&12&12&12&12&10&10&12&$1.36285\cdot10^{-8}$\\\hline
	3&12&12&12&12&10&10&8&$1.25704\cdot10^{-8}$\\\hline
	4&12&12&12&12&10&10&6&$1.2035\cdot10^{-8}$\\\hline
	5&12&12&12&12&10&10&4&$3.97122\cdot10^{-8}$\\\hline
	6&12&12&12&12&10&12&6&$1.24734\cdot10^{-8}$\\\hline
	7&12&12&12&12&10&8&6&$1.15079\cdot10^{-8}$\\\hline
	8&12&12&12&12&10&6&6&$1.13142\cdot10^{-8}$\\\hline
	9&12&12&12&12&10&4&6&$1.10034\cdot10^{-8}$\\\hline
	10&12&12&12&12&10&2&6&$4.68726\cdot10^{-3}$\\\hline
	11&12&12&12&12&12&4&6&$1.20275\cdot10^{-10}$\\\hline
	12&12&12&12&12&14&4&6&$4.27822\cdot10^{-11}$\\\hline
	13&12&12&12&12&16&4&6&$5.42373\cdot10^{-11}$\\\hline
	14&12&12&12&14&14&4&6&$4.28631\cdot10^{-11}$\\\hline
	15&12&12&12&10&14&4&6&$4.2808\cdot10^{-11}$\\\hline
	16&12&12&14&12&14&4&6&$4.29279\cdot10^{-11}$\\\hline
	17&12&12&10&12&14&4&6&$4.26967\cdot10^{-11}$\\\hline
	18&12&12&8&12&14&4&6&$4.25943\cdot10^{-11}$\\\hline
	19&12&12&6&12&14&4&6&$4.24758\cdot10^{-11}$\\\hline
	20&12&12&4&12&14&4&6&$4.25071\cdot10^{-11}$\\\hline
	21&12&14&6&12&14&4&6&$4.24995\cdot10^{-11}$\\\hline
	22&12&10&6&12&14&4&6&$4.23524\cdot10^{-11}$\\\hline
	23&12&8&6&12&14&4&6&$4.17446\cdot10^{-11}$\\\hline
	24&12&6&6&12&14&4&6&$4.35334\cdot10^{-11}$\\\hline
	25&14&8&6&12&14&4&6&$6.80995\cdot10^{-12}$\\\hline
	26&16&8&6&12&14&4&6&$6.58355\cdot10^{-12}$\\\hline
	27&18&8&6&12&14&4&6&$6.5848\cdot10^{-12}$\\\hline
	\end{tabular}
	\caption{Mean squared errors for $U_\theta$ and bandwidths $\Nbf^\ubf=(\betr{\{i\}\cap \ubf}N_{\betr{\ubf}})_{i=1}^4$ for some choices of the parameters $N_1$ and $N_2$.}
	\label{tb:1}
\end{table}
We get the bandwidths $\Nbf^\ubf=(\betrk{\{i\}\cap\ubf}N_{\ubf})_{i=1}^4$ with the parameters $N_{\{1\}}=16$, $N_{\{2\}}=8$, $N_{\{3\}}=6$, $N_{\{4\}}=12$, $N_{\{1,2\}}=14$, $N_{\{2,4\}}=4$, and $N_{\{3,4\}}=6$. Finally, we repeat the one by one optimizing procedure again with all parameters for the bandwidths, e.g. we consider every non-zero entry of each bandwidth as one parameter. As result, we get the bandwidths
\begin{align*}
	&\Nbf^{\{1\}}=
	\begin{psmallmatrix}
		16\\0\\0\\0
	\end{psmallmatrix},&&\;
	\Nbf^{\{2\}}=
	\begin{psmallmatrix}
		0\\8\\0\\0
	\end{psmallmatrix},&&\;
	\Nbf^{\{3\}}=
	\begin{psmallmatrix}
		0\\0\\2\\0
	\end{psmallmatrix},&&\;
	\Nbf^{\{4\}}=
	\begin{psmallmatrix}
		0\\0\\0\\10
	\end{psmallmatrix},&&\\
	&\Nbf^{\{1,2\}}=
	\begin{psmallmatrix}
		16\\8\\0\\0
	\end{psmallmatrix},&&\;
	\Nbf^{\{2,3\}}=
	\begin{psmallmatrix}
		0\\2\\4\\0
	\end{psmallmatrix},\text{ and }&&
	\Nbf^{\{2,4\}}=
	\begin{psmallmatrix}
		0\\8\\0\\8
	\end{psmallmatrix}&&&&
\end{align*}
with a mean squared error of $9.74704\cdot10^{-14}$. We repeat the procedure $N_{\text{MSE}}=100$ times with new randomly distributed data points to account for the variance. We show the resulting MSEs in Figure \ref{dg:7}.
\begin{figure}
	\center
	\begin{tikzpicture}
	\begin{axis}[width=\linewidth,
	boxplot/draw direction=y,
	axis x line*=bottom,
	axis y line*=left,
	xtick={1},
	xticklabels={Mixed basis},
	ylabel={${\textrm{MSE}}$},
	boxplotcolor/.style={color=#1,fill=#1!70,mark options={color=#1,fill=#1!70}},width=3.75cm,height=7.5cm
	]
	\addplot+[boxplotcolor=black,boxplot prepared={
		lower whisker=9.742839857268769e-14, lower quartile=1.5834737850918943e-13,
		median=2.4547414100803726e-13, upper quartile=3.800961670899178e-13,
		upper whisker=7.127193499610104e-13}]
	coordinates {(0, 7.237988512143848e-13)};
	\end{axis}
	\begin{axis}[width=\linewidth,at={(1700,0)},
	boxplot/draw direction=y,
	axis x line*=bottom,
	axis y line*=left,
	xtick={1},
	xticklabels={Cosine basis},
	ylabel={${\textrm{MSE}}$},
	boxplotcolor/.style={color=#1,fill=#1!70,mark options={color=#1,fill=#1!70}},width=3.75cm,height=7.5cm
	]
	\addplot+[boxplotcolor=blue,mark=square*,boxplot prepared={
		lower whisker=3.4782162479853885e-5, lower quartile=4.198560716357893e-5,
		median=4.681238643468041e-5, upper quartile=5.229767340715227e-5,
		upper whisker=6.776577277251228e-5}]
	coordinates {(0, 7.603123293866874e-5)(0, 7.782806544739678e-5)};
	\end{axis}
	\end{tikzpicture}
	\caption{Boxplot of the MSEs for the approximation of $f_1$ with mixed and cosine basis.}\label{dg:7}
\end{figure}
We use standard box plots in this figure, i.e. the lower and upper boundaries of the box represent the first and third quantiles and the whiskers have a maximum length of 1.5 times the interquantile range. We point out that median of the MSEs is $2.45474\cdot10^{-14}$. All in all, we approximate the function $f_1$ with a sum of 365 basis functions combined with the same number of coefficients. Next, we approximate the function with the cosine basis. We do this the same way as it is described above for the mixed basis. We end up with the same truncation set $U_\theta$. Again we show the MSEs for $N_{\text{MSE}}=100$ runs in Figure \ref{dg:7}. We obtain an median of $4.68123\cdot10^{-5}$ which is significantly worse in comparison to the mixed basis. It should be noted that the approximation with the mixed basis is orders of magnitude better than the approximation with the cosine basis.

\subsection{Comparison of analytic and approximated global sensitivity indices}\label{sec:5.2}
In this subsection we approximate a function multiple times with different numbers of nodes $M$. This time we restrict ourselves to uniformly sampled nodes. This has the advantage that we can compare the ANOVA approximation with the mixed basis to the Fourier basis and with the half period cosine basis approximation. We also compare the approximated GSIs with analytically calculated ones. In order to do this, we consider the function
\begin{align*}
\fkt{f_2}{[0,1]^4}{\Cbb}{x_1,x_2,x_3,x_4}{(2x_1-1)^2 x_3+10\sin(2\pi x_1)\left(x_2-\frac{1}{2}\right)^2+\exp(x_3)}.
\end{align*}
The function $f_2$ does not depend on the variable $x_4$. Furthermore, the function has the same values at the boundaries in dimension one and two, e.g.
\begin{align*}
f_2(0,x_2,x_3,x_4) &= f_2(1,x_2,x_3,x_4),\;\forall x_2,x_3,x_4\in[0,1]\andt\\
f_2(x_1,0,x_3,x_4) &= f_2(x_1,1,x_3,x_4),\;\forall x_1,x_3,x_4\in[0,1].
\end{align*}
Thus, we should use the Fourier basis for the first two coordinates and the half period cosine basis for the third, i.e. $\dbf_2\coloneqq(\exp,\exp,\cos,\cos)^\top$. Furthermore, we test two more ANOVA approximations without mixed bases, namely one with a Fourier basis and one with a half period cosine basis. In the appendix \ref{c:1} we calculate the analytic GSIs of this function $f_2$. The results are
\begin{alignat*}{3}
&\rho(\{1\},f_2)&&=\frac{133}{59+600\e-180\e^2}&&\approx0.369507,\\
&\rho(\{3\},f_2)&&=\frac{-530+1800\e-540\e^2}{177+1800\e-540\e^2}&&\approx0.345259,\\
&\rho(\{1,2\},f_2)&&=\frac{100}{59+600\e-180\e^2}&&\approx0.277825,\\
&\rho(\{1,3\},f_2)&&=\frac{8}{177+1800\e-540\e^2}&&\approx0.007409,
\end{alignat*}
and the other analytic GSIs are zero.\\
We now compare this with the approximated GSIs. For ANOVA approximation we use $M=50,100,200,500,1000,2000,5000,10000,20000$ and $50000$ uniformly distributed nodes $\Xcal$. For this function $f_2$ we consider the superposition set $U_{d_s}$ with $d_s=2$, since the function $f_2$ has only one-dimensional and two dimensional interactions between variables. In this example we restrict ourselves to two bandwidth parameters, $N_1$ for one-dimensional bandwidths and $N_2$ for the two-dimensional bandwidths. For $M\leq10000$ we determine the bandwidth parameters $N_1$ and $N_2$ numerically like in \eqref{eq:29}. The results are shown in the Table \ref{tb:3}.
\begin{table}
	\small
	\centering
	\begin{tabular}[h]{|R{3em}||R{2.5em}|R{2.5em}||R{2.5em}|R{2.5em}||R{2.5em}|R{2.5em}|}
	\hline
	\multirow{2}{*}{$M$}&\multicolumn{2}{c||}{$\cosbf$}&\multicolumn{2}{c||}{$\dbf_2$}&\multicolumn{2}{c|}{$\expbf$}\\
	&$N_1$&$N_2$&$N_1$&$N_2$&$N_1$&$N_2$\\
	\hline
	50&4&2&4&2&4&2\\
	\hline
	100&4&4&4&4&4&4\\
	\hline
	200&6&4&6&4&14&4\\
	\hline
	500&12&8&10&8&32&6\\
	\hline
	1000&18&10&14&10&76&6\\
	\hline
	2000&28&14&24&14&150&10\\
	\hline
	5000&56&22&40&22&300&14\\
	\hline
	10000&70&32&60&32&720&18\\
	\hline\hline
	20000&170&46&110&46&1962&26\\
	\hline
	50000&382&76&224&76&6548&40\\
	\hline
	\end{tabular}
	\caption{Optimal bandwidths (for $M\leq10000$) $\Nbf^\ubf=(\betr{\{i\}\cap \ubf}N_{\betr{\ubf}})_{i=1}^4$ for $U_2$ for $f_2$ approximated at $M$ training nodes and (for $M>10000$) extrapolated bandwidths.}
	\label{tb:3}
\end{table}
We obtain the bandwidths for $M>10000$ by extrapolating the previously determined optimal bandwidths.\\
In Figure \ref{dg:3} we plot the resulting mean squared errors. Here we notice that the error for the approximation with the Fourier basis decays with the rate $M^{-1}\log(M)$. The error of the approximation with the mixed basis and the half period cosine basis decays with the rate $M^{-\frac{3}{2}}\log(m)^{-\frac{3}{2}}$, while the approximation with the mixed basis gives a better constant.\\

In the next part we show that the observed decay rate for the approximation with the mixed basis is optimal. The mixed coefficients of the two-dimensional ANOVA term $f_{\{1,3\}}$ decay quadratically in both directions, i.e. $\betrk{c_\kbf^{\dbf_2}(f_{\{1,3\}})}\leq Ck_1^{-2}k_2^{-2},\;C>0$. For the calculation of these mixed coefficients, see Appendix \ref{c:1}. We consider the error of the projection of the ANOVA term $f_{\{1,3\}}$ onto the set of polynomials $\Tcal_{\Ical_\Nbf^{\dbf_3}}(\Bcal^{\dbf_3})$ with $\Nbf\coloneqq(N,N)^\top$ and $\dbf_3\coloneqq(\exp,\cos)^\top$. We obtain
\begin{align*}
	\norm{f_{\{1,3\}}-P_{\Tcal_{\Ical_\Nbf^{\dbf_3}}(\Bcal^{\dbf_3})}f_{\{1,3\}}}^2_{\Lp{2}(\Dbb^{\dbf_3},\omega^{\dbf_3})}&=\sum_{\kbf\notin\Ical_\Nbf^{\dbf_3}}\betrk{c_\kbf^{\dbf_2}(f_{\{1,3\}})}^2\\
	&\leq C\sum_{\kbf\notin\Ical_\Nbf^{\dbf_3}}(k_1^{-2}k_2^{-2})^2\\
	&={}C\sum_{k_1=1}^N\sum_{k_2=N+1}^\infty(k_1^{-2}k_2^{-2})^2+C\sum_{k_1=N+1}^\infty\sum_{k_2=N+1}^\infty(k_1^{-2}k_2^{-2})^2\\
	&={}C\left(\sum_{k_1=1}^Nk_1^{-4}\right)\left(\sum_{k_2=N+1}^\infty k_2^{-4}\right)+C\left(\sum_{k=N+1}^\infty k^{-4}\right)\\
	&={}C_1 N^{-3}+C_2N^{-6}\\
\end{align*}
with constants $C_1>0$ and $C_2>0$. So the projection error is in $\Ocal(N^{-3})$. Since we need logarithmic oversampling, see \cite{Ba23}, we have $N^2\in\Ocal(M\log(M)^{-1})$ such that we get a projection error of $\Ocal(M^{-\frac{3}{2}}\log(M)^{\frac{3}{2}})$. This is an lower bound to the error of the approximation $\normk{f-\tilde{f}}^2_{\Lp{2}(\Dbb^{\dbf_2},\omega^{\dbf_2})}$. Similar calculations show that the decay rates for the errors of the projections of the other ANOVA terms are faster. Furthermore, the MSE approximates the error $\normk{f-\tilde{f}}^2_{\Lp{2}(\Dbb^{\dbf_2},\omega^{\dbf_2})}$ when the test set $\Xcal_\textrm{test}$ is distributed in $\Dbb^{\dbf_2}$ with the density $\omega^{\dbf_2}$. In short, we see the expected decay rate of the MSE at the approximation with the mixed basis. For the approximation with the cosine basis hold similar calculations, since the cosine coefficients of the function $c_\kbf^{\cosbf}(f_2)$ decay in a similar way. So we get the same decay rate for the approximation with the cosine basis. The rate of the approximation with the Fourier basis is slower because of the Fourier coefficients of the function $c_\kbf^{\expbf}(f_2)$ have lower decay rate, e.g. $\betrk{c_k^{\exp}(f_{\{3\}})}=Ck^{-1}$, $C>0$. So we cannot expect a decay rate better than $M^{-1}\log(M)$.\\
Next, we use the approximation again with the cosine basis to compare it with the mixed basis. 
\begin{figure}
	\begin{minipage}[b]{.48\linewidth}
		\centering
	\begin{tikzpicture}
	\begin{axis}[name=boundary,axis x line=bottom,axis y line=left,ymin=0.000004,xmax=0,xmode=log,ymode=log,width=7.5cm,height=7.5cm,grid=minor,x label style={at={(axis description cs:0.5,0)},anchor=north},y label style={at={(axis description cs:0.04,0.5)},anchor=south},xlabel={$M$}, ylabel={$\mop{MSE}(f_2,\tilde{f}_2,\Xcal_\textrm{test})$}]
	\addplot[dash dot,mark options={fill=black,solid},mark=triangle*, color=black] table [x=a, y=b, col sep=space] {data1.dat};\label{pl:3.1}
	\addplot[dashed,mark options={fill=violet,solid}, mark=diamond*, color=violet] table [x=a, y=h, col sep=space] {data1.dat};\label{pl:3.2}
	\addplot[dash dot dot,mark options={fill=blue,scale=0.7,solid}, mark=*, color=blue] table [x=a, y=d, col sep=space] {data1.dat};\label{pl:3.3}
	\addplot[domain = 50:50000, color = violet] {7/x^1.66*ln(x)^1.66};\label{pl:3.4}
	\addplot[domain = 50:50000, color = blue] {6/x^1.1*ln(x)^1.1};\label{pl:3.5}
	\end{axis}
	\node[draw,fill=white,inner sep=0pt] at (4.2,5.5) {\small
		\begin{tabular}{cl}
		\ref{pl:3.1} & {$\cosbf$}\\
		\ref{pl:3.2} & {$\dbf_2$}\\
		\ref{pl:3.3} & {$\expbf$}\\
		\ref{pl:3.4} & {$7M^{-\frac{3}{2}}\log(M)^{-\frac{3}{2}}$}\\
		\ref{pl:3.5} & {$6M^{-1}\log(M)$}\\
		\end{tabular}};
	\end{tikzpicture}
	\caption{Mean squared errors of the ANOVA approximation $\tilde{f}_2$ to $f_2$ for the bandwidths from Table \ref{tb:3} and $M$ training nodes, evaluated at 10000 nodes.}\label{dg:3}
	\end{minipage}
	\begin{minipage}[b]{.48\linewidth}
		\centering
		\begin{tikzpicture}
		\begin{axis}[name=boundary,axis x line=bottom,axis y line=left,xmode=log,ymode=log,width=7.5cm,height=7.5cm,grid=minor,x label style={at={(axis description cs:0.5,0)},anchor=north},y label style={at={(axis description cs:0.03,0.5)},anchor=south},xlabel={$M$}, ylabel={$\left(\sum_{\ubf\in U}(\rho(\ubf,f_2)-\rho(\ubf,\tilde{f}_2))^2\right)^\frac{1}{2}$}]
		\addplot[dash dot,mark options={fill=black,solid}, mark=triangle*, color=black] table [x=a, y=e, col sep=space] {data1.dat};\label{pl:4.1}
		\addplot[dash dot dot,mark options={fill=violet,solid}, mark=diamond*, color=violet] table [x=a, y=f, col sep=space] {data1.dat};\label{pl:4.2}
		\addplot[dashed,mark options={fill=blue,scale=0.7,solid}, mark=*, color=blue] table [x=a, y=g, col sep=space] {data1.dat};\label{pl:4.3}
		\end{axis}
		\node[draw,fill=white,inner sep=0pt] at (4.87,5.3) {\small
			\begin{tabular}{cl}
			\ref{pl:4.1} & {$\cosbf$}\\
			\ref{pl:4.2} & {$\dbf_2$}\\
			\ref{pl:4.3} & {$\expbf$}\\
			\end{tabular}};
		\end{tikzpicture}
		\caption{Deviation of approximated global sensitivity indices for bandwidths from Table \ref{tb:3} and $M$ training nodes to the analytic GSIs.}\label{dg:4}
\end{minipage}
\end{figure}
In Figure \ref{dg:4}, we compare the approximated GSIs with the analytic GSIs and see that they converge. The approximated GSIs using the Fourier basis converge slower than the approximated GSIs using the mixed basis. In Figure \ref{dg:5} we consider the individual approximated GSIs for different numbers of training nodes. Here we notice that the approximated GSIs using the Fourier basis performs particularly poorly in the dimensions where the function $f_2$ is not continuously periodizable, e.g. for $\ubf=\{3\}$ we have particularly large deviations from the analytic GSI. Furthermore, we see for example at $\ubf=\{1,2\}$ that the approximated GSIs using the half period cosine basis converge more slowly towards the analytic GSI than approximated GSIs using the Fourier basis. The approximated GSIs using the mixed basis combines the positive properties of the other two ANOVA approximations and therefore converges much faster.
\begin{figure}
	\centering
	\begin{tikzpicture}
	\begin{axis}[name=boundary,ymin=0,ymax=0.5,axis y line=left,axis x line*=bottom,width=17cm,height=8cm,symbolic x coords/.expand once={$\{1\}$,$\{2\}$,$\{3\}$,$\{4\}$,$\{1;2\}$,$\{1;3\}$,$\{1;4\}$,$\{2;3\}$,$\{2;4\}$,$\{3;4\}$},xtick=data,x tick label style={rotate=90,anchor=east},grid=major,xticklabel style = {yshift=-20.5},]
	\addplot+[ycomb,mark=triangle*,shift={(7.7,0)}, color=black, dashdotted,mark options={fill=black,solid}] table [x=a, y=j, col sep=space] {data2.dat};\label{pl:5.10}
	\addplot+[ycomb,mark=diamond*,shift={(15.4,0)}, color=black, dashdotted,mark options={fill=black,solid}] table [x=a, y=k, col sep=space] {data2.dat};\label{pl:5.11}
	\addplot+[ycomb,mark=*,shift={(23.1,0)}, color=black, dashdotted,mark options={fill=black,scale=0.7,solid}] table [x=a, y=l, col sep=space] {data2.dat};\label{pl:5.12}
	\addplot+[ycomb,mark=star,shift={(30.8,0)}, color=black, dashdotted] table [x=a, y=m, col sep=space] {data2.dat};\label{pl:5.13}
	\addplot+[ycomb,mark=triangle*,shift={(38.5,0)}, color=violet, solid,mark options={fill=violet}] table [x=a, y=b, col sep=space] {data2.dat};\label{pl:5.1}
	\addplot+[ycomb,mark=diamond*,shift={(46.2,0)}, color=violet, solid,mark options={fill=violet}] table [x=a, y=d, col sep=space] {data2.dat};\label{pl:5.3}
	\addplot+[ycomb,mark=*,shift={(53.9,0)}, color=violet, solid,mark options={fill=violet,scale=0.7}] table [x=a, y=f, col sep=space] {data2.dat};\label{pl:5.5}
	\addplot+[ycomb,mark=star,shift={(61.6,0)}, color=violet, solid] table [x=a, y=h, col sep=space] {data2.dat};\label{pl:5.7}
	\addplot+[ycomb,mark=triangle*, shift={(69.3,0)}, color=blue, dashed,mark options={fill=blue,solid}] table [x=a, y=c, col sep=space] {data2.dat};\label{pl:5.2}
	\addplot+[ycomb,mark=diamond*, shift={(77,0)}, color=blue, dashed,mark options={fill=blue,solid}] table [x=a, y=e, col sep=space] {data2.dat};\label{pl:5.4}
	\addplot+[ycomb,mark=*, shift={(84.7,0)}, color=blue, dashed,mark options={fill=blue,scale=0.7,solid}] table [x=a, y=g, col sep=space] {data2.dat};\label{pl:5.6}
	\addplot+[ycomb,mark=star, shift={(92.4,0)}, color=blue, dashed] table [x=a, y=i, col sep=space] {data2.dat};\label{pl:5.8}
	\addplot[red,sharp plot,update limits=false,xscale=1,shift={(0,0)}] coordinates {($\{1\}$,0.369507) ($\{2\}$,0.369507)};\label{pl:5.9}
	\addplot[red,sharp plot,update limits=false,xscale=1,shift={(0,0)},very thick] coordinates {($\{2\}$,0) ($\{3\}$,0)};
	\addplot[red,sharp plot,update limits=false,xscale=1,shift={(0,0)}] coordinates {($\{3\}$,0.345259) ($\{4\}$,0.345259)};
	\addplot[red,sharp plot,update limits=false,xscale=1,shift={(0,0)},very thick] coordinates {($\{4\}$,0) ($\{1;2\}$,0)};
	\addplot[red,sharp plot,update limits=false,xscale=1,shift={(0,0)}] coordinates {($\{1;2\}$,0.277825) ($\{1;3\}$,0.277825)};
	\addplot[red,sharp plot,update limits=false,xscale=1,shift={(0,0)}] coordinates {($\{1;3\}$,0.007409) ($\{1;4\}$,0.007409)};
	\addplot[red,sharp plot,update limits=false,xscale=1.4,shift={(-171.5,0)},very thick] coordinates {($\{1;4\}$,0) ($\{3;4\}$,0)};
	\end{axis}\node[draw,fill=white,inner sep=0pt] at (12,3.5) {\small
	\begin{tabular}{cll}
		\ref{pl:5.10} & \multirow{4}{*}{$\cosbf$}&$M=50$\\
		\ref{pl:5.11} & &$M=100$\\
		\ref{pl:5.12} & &$M=200$\\
		\ref{pl:5.13} & &$M=500$\\
		\hline
		\ref{pl:5.1} & \multirow{4}{*}{$(\exp,\exp,\cos,\cos)^\top$}&$M=50$\\
		\ref{pl:5.3} & &$M=100$\\
		\ref{pl:5.5} & &$M=200$\\
		\ref{pl:5.7} & &$M=500$\\
		\hline
		\ref{pl:5.2} & \multirow{4}{*}{$\expbf$}&$M=50$ \\
		\ref{pl:5.4} & &$M=100$ \\
		\ref{pl:5.6} & &$M=200$ \\
		\ref{pl:5.8} & &$M=500$ \\
		\hline
		\ref{pl:5.9}&\multicolumn{2}{l}{Analytic GSIs}\\
		\end{tabular}};
	\end{tikzpicture}
	\caption{Global sensitivity indices for $U_2$ and bandwidths from Table \ref{tb:3} for $f_2$ trained on $M$ randomly chosen nodes.}\label{dg:5}
\end{figure}
\subsection{Numerical Experiment with Real Data}\label{sec:5.3}
In this subsection we use the airfoil self-noise dataset. This dataset was used in \cite{PoSc21}. In that paper it was approximated using the ANOVA approximation with the cosine basis and the results were compared with results from the literature using different machine learning methods. It was found that the ANOVA approximation gave the best results, and it is interpretable. This dataset is about NACA airfoils tested in wind tunnels with different wind speeds and angles of attack tested by the NASA. The goal is to predict the scaled sound pressure level of the self-noise in decibels, see \cite{UCIr}. The dataset consists of 1503 nodes, each with 5 attributes. We split the data in a training set $\Xcal$ of 80\% and a test set $\Xcal_{\text{test}}$ of 20\% like it is done in \cite{PoSc21} such that we can compare the results. Furthermore, we normalize the nodes into $[0,1]$.\\
We use the same truncation set as \cite{PoSc21}, namely $U=\{\{1\},\{2\},\{3\},\{4\},\{5\},\{1,2\},\{1,3\},\{1,4\},\{1,5\},$ $\{2,3\},\{2,5\},\{3,4\},\{3,5\},\{4,5\}\}$. In each dimension we test the three different bases and choose the one that gives the smallest error. As result we get the Fourier basis in the first two dimensions, the Chebyshev basis in the third and fourth dimension and the cosine basis in the last dimension. This is interpretable, i.e. it shows that the first two dimensions have a more periodic behaviour than the last dimension, but further interpretation requires more knowledge of the physical background and the technical details of the experiment. The chosen Chebyshev basis in the third and fourth dimensions suggests that the density of the sampled points is closer to the Chebyshev density than to the uniform density. We use bandwidths of the form $\Nbf^\ubf=(\betr{\{i\}\cap \ubf}N_{\betr{\ubf}})_{i=1}^5$ with bandwidth parameters $N_{\ubf}$ for $\ubf\in U$. We choose these bandwidth parameters $N_{\ubf}$ for $\ubf\in U$ through cross validation. To this end, we take 100 random data splits and obtain models for these training sets $\Xcal$. Next, we determine the error in the models using the corresponding test sets $\Xcal_{\text{test}}$. We get 3.72 as median of the relative errors with the model with the mixed basis. In comparison, we get 4.21 as median of the relative errors with the model as in \cite{PoSc21}. This yields an improvement of 11.6\%. An boxplot of the results is shown in Figure \ref{dg:6}.
\begin{figure}
	\center
		\begin{tikzpicture}
		\begin{axis}[width=\linewidth,
		boxplot/draw direction=y,
		axis x line*=bottom,
		axis y line*=left,
		xtick={1,2},
		xticklabels={Mixed basis,Cosine basis},
		ylabel={MSE},
		boxplotcolor/.style={color=#1,fill=#1!70,mark options={color=#1,fill=#1!70}},width=7.5cm,height=7.5cm
		]
		\addplot+[boxplotcolor=black,boxplot prepared={
			lower whisker=2.4385422949328275, lower quartile=3.4522976794538005,
			median=3.72095160035243, upper quartile=4.133838070568107,
			upper whisker=4.743262187023889}]
		coordinates {(0,4.792124772899251)(0,5.2457573508029665)(0,4.817773505905978)(0,5.046917755744663)(0,5.353193940116853)};
		\addplot+[boxplotcolor=blue,boxplot prepared={
			lower whisker=3.3544788092313857, lower quartile=3.780207117306852,
			median=4.211072093694923, upper quartile=4.752459942110381,
			upper whisker=5.669451330900216}]
		coordinates {(0,6.9499173015921185)(0,5.835736412343023)(0,5.70059478509507)(0,5.791393500891101)};
		\end{axis}
		\end{tikzpicture}
		\caption{Boxplot of the relative errors of the models for the airfoil self-noise dataset with mixed and cosine basis.}\label{dg:6}
\end{figure}

\begin{figure}
	
\end{figure}
\markboth{Bibliography}{Bibliography}
\printbibliography
\appendix
\section*{Acknowledgement}
This work is dedicated to the 70th birthday of Albrecht Böttcher. It has
been a great honer for DP to work with Albrecht. He would especially
like to mention the joint work on publications, as well as the joint
organization of many summer schools in Chemnitz. PS got to know Albrecht
Böttcher through outstanding lectures in linear algebra and functional
analysis.  With his lively, precise lectures, Albrecht not only
influenced several generations of young students but also the Department of
Mathematics in Chemnitz for many years.\newline\newline
PS gratefully acknowledge support by the German Ministry for Economic Affairs and Climate Protection for funding the project 16KN092847 as part of the "Central Innovation Programme for SMEs" (ZIM).
\section{Analytic Calculation of Global Sensitivity Indices}\label{c:1}
In the following we calculate the analytic GSIs for the function
\begin{alignat*}{2}
	\fktl{f_2}{\Dbb^{\dbf_2}}{\Cbb}{x_1,x_2,x_3,x_4}{(2x_1-1)^2 x_3+10\sin(2\pi x_1)\left(x_2-\frac{1}{2}\right)^2+\exp(x_3)}\\
	\fktlnl{4x_1^2x_3-4x_1x_3+x_3+10\sin(2\pi x_1)x_2^2-10\sin(2\pi x_1)x_2+\frac{5}{2}\sin(2\pi x_1)+\exp(x_3)}
\end{alignat*}
with $\dbf_2=(\exp,\exp,\cos,\cos)^\top$ using the basis $\Bcal^{\dbf_2}$.
First we calculate the mixed coefficients $c^{\dbf_2}_\kbf(f_2)$, exploiting linearity. For this purpose we define
\begin{align*}
&h_j\colon\Dbb^{\dbf_2}\rightarrow\Cbb,\;j=1,\ldots{},7,\\
&h_1(x_1,x_2,x_3,x_4)=x_1^2x_3,\\
&h_2(x_1,x_2,x_3,x_4)=x_1x_3,\\
&h_3(x_1,x_2,x_3,x_4)=x_3,\\
&h_4(x_1,x_2,x_3,x_4)=\sin(2\pi x_1)x_2^2,\\
&h_5(x_1,x_2,x_3,x_4)=\sin(2\pi x_1)x_2,\\
&h_6(x_1,x_2,x_3,x_4)=\sin(2\pi x_1),\text{ and }\\
&h_7(x_1,x_2,x_3,x_4)=\exp(x_3).
\end{align*}
and observe
\begin{align}
c^{\dbf_2}_\kbf(f_2)= 4c^{\dbf_2}_\kbf(h_1)-4c^{\dbf_2}_\kbf(h_2)+c^{\dbf_2}_\kbf(h_3)+10c^{\dbf_2}_\kbf(h_4)-10c^{\dbf_2}_\kbf(h_5)+\frac{5}{2}c^{\dbf_2}_\kbf(h_6)+c^{\dbf_2}_\kbf(h_7)\label{eq:a6}
\end{align}
for all $\kbf\in\Kbb^{\dbf_2}$.
We know that, if $f\in\Lp{2}(\Dbb^\dbf)$ is a function given as product $f(\xbf)=\prod_{j=1}^df^{d_j}_j(x_j)$ of functions $f^{d_j}_j\in\Lp{2}(\Dbb^{d_j})$, $j=1,\ldots{},d$, then for all $\kbf\in\Kbb^\dbf$ we can decompose the mixed coefficients, i.e.
\begin{align*}
c^\dbf_\kbf(f)=\prod_{j=1}^dc^{d_j}_{k_j}(f^{d_j}_j(x_j)).
\end{align*}
Thus, we decompose the functions $h_i$ into
\begin{align*}
&g^{\exp}_j\colon\Dbb^{\exp}\rightarrow\Cbb,\;j=1,\ldots{},4,\;\;\;&&g^{\cos}_j\colon\Dbb^{\cos}\rightarrow\Cbb,\;j=1,\ldots{},3,\\
&g^{\exp}_1(x)=1,&&g^{\cos}_1(x)=1,\\
&g^{\exp}_2(x)=x,&&g^{\cos}_2(x)=x,\\
&g^{\exp}_3(x)=x^2,&&g^{\cos}_3(x)=\exp(x),\\
&g^{\exp}_4(x)=\sin(2\pi x).&&
\end{align*}
Next we calculate the Fourier coefficients and the cosine coefficients of these functions. We observe $c^{\exp}_k(g^{\exp}_1)=\delta_{k,0}$ and $c^{\cos}_k(g^{\cos}_1)=\delta_{k,0}$ because of the orthogonality of the basis functions and $\phi_0^{\exp}=\phi_0^{\cos}=1$. We start with the zeroth Fourier coefficients and cosine coefficients and observe
\begin{align*}
&c^{\exp}_0(g^{\exp}_2)=c^{\cos}_0(g^{\cos}_2)=\frac{1}{2},&&c^{\exp}_0(g^{\exp}_3)=\frac{1}{3}&&\text{ and }&&c^{\cos}_0(g^{\cos}_3)=\e-1
\end{align*}
through $c_0^{\exp}=c_0^{\cos}=\iglse{0}{1}{f(x)}{x}$. For the case $k\neq0$ we observe 
\begin{align*}
&c^{\exp}_k(g^{\exp}_2)=\frac{1}{2\pi k}\im,&&&&c^{\cos}_k(g^{\cos}_2)=\sqrt{2}\frac{(-1)^k-1}{\pi^2k^2},\\
&c^{\exp}_k(g^{\exp}_3)=\frac{1}{2\pi^2 k^2}+\frac{1}{2\pi k}\im&&\text{ and }&&c^{\cos}_k(g^{\cos}_3)=\sqrt{2}\frac{(-1^k)\e-1}{\pi^2k^2+1}.
\end{align*}
To this end we use the identity $g^{\exp}_4(x)=\sin(2\pi x)=\frac{1}{2\im}\exp(2\pi\im x)-\frac{1}{2\im}\exp(-2\pi\im x)$ for the coefficients $c^{\exp}_k(g^{\exp}_4)$ and get $c^{\exp}_1(g^{\exp}_4)=\frac{1}{2\im}$, $c^{\exp}_{-1}(g^{\exp}_4)=-\frac{1}{2\im}$, and the ohter coefficients are zero. Next, we consider the Fourier cosine coefficients $c_\kbf^{\dbf_2}$ of the functions $h_i$, $i=1,\ldots{},7$ and obtain
\begin{align*}
c^{\dbf_2}_\kbf(h_1)&=\begin{cases}
\frac{1}{6},&k_1,k_2,k_3,k_4=0\\
\frac{1}{4\pi^2 k_1^2}+\frac{1}{4\pi {k_1}}\im,&k_2,k_3,k_4=0,k_1\neq0\\
\sqrt{2}\frac{(-1)^{k_3}-1}{3\pi^2k_3^2},&k_1,k_2,k_4=0,k_3\neq0\\
\sqrt{2}\frac{(-1)^{k_3}-1}{2\pi^4k_1^2k_3^2}+\sqrt{2}\frac{(-1)^{k_3}-1}{2\pi^3k_1k_3^2},&k_2,k_4=0,k_1,k_3\neq0\\
0,&\text{else}
\end{cases},\\
c^{\dbf_2}_\kbf(h_2)&=\begin{cases}
\frac{1}{4},&k_1,k_2,k_3,k_4=0\\
\frac{1}{4\pi k_1}\im,&k_2,k_3,k_4=0,k_1\neq0\\
\sqrt{2}\frac{(-1)^{k_3}-1}{2\pi^2k_3^2},&k_1,k_2,k_4=0,k_3\neq0\\
1\sqrt{2}\frac{(-1)^{k_3}-1}{2\pi^3k_1k_3^2}\im,&k_2,k_4=0,k_1,k_3\neq0\\
0,&\text{else}
\end{cases},\\
c^{\dbf_2}_\kbf(h_3)&=\begin{cases}
\frac{1}{2},&k_1,k_2,k_3,k_4=0\\
\sqrt{2}\frac{(-1)^{k_3}-1}{\pi^2k_3^2},&k_1,k_2,k_4=0,k_3\neq0\\
0,&\text{else}
\end{cases},\\
c^{\dbf_2}_\kbf(h_4)&=\begin{cases}
-\frac{\im}{6},&k_2,k_3,k_4=0,k_1=1\\
\frac{\im}{6},&k_2,k_3,k_4=0,k_1=-1\\
\frac{1}{4\pi k_1}-\frac{1}{2\pi^2 k_1^2}\im,&k_3,k_4=0,k_1=1,k_2\neq0\\
-\frac{1}{4\pi k_1}+\frac{1}{2\pi^2 k_1^2}\im,&k_3,k_4=0,k_1=-1,k_2\neq0\\
0,&\text{else}
\end{cases},\\
c^{\dbf_2}\kbf(h_5)&=\begin{cases}
-\frac{\im}{4},&k_2,k_3,k_4=0,k_1=1\\
\frac{\im}{4},&k_2,k_3,k_4=0,k_1=-1\\
\frac{1}{4\pi k_2},&k_3,k_4=0,k_1=1,k_2\neq0\\
-\frac{1}{4\pi k_2},&k_3,k_4=0,k_1=-1,k_2\neq0\\
0,&\text{else}
\end{cases},\\
c^{\dbf_2}_\kbf(h_6)&=\begin{cases}
-\frac{\im}{2},&k_2,k_3,k_4=0,k_1=1\\
\frac{\im}{2},&k_2,k_3,k_4=0,k_1=-1\\
0,&\text{else}
\end{cases},\text{ and }\\
c^{\dbf_2}_\kbf(h_7)&=\begin{cases}
\e-1,&k_1,k_2,k_3,k_4=0\\
\sqrt{2}\frac{(-1)^{k_3}\e-1}{\pi^2k_3^2+1},&k_1,k_2,k_4=0,k_3\neq0\\
0,&\text{else}
\end{cases}.
\end{align*}
Finally, using the \eqref{eq:a6} we calculate the Fourier cosine coefficients $c^{\dbf_2}_\kbf(f_2)$ for the function $f_2$,
\begin{align*}
c^{\dbf_2}_\kbf(f_2)
&=\begin{cases}
\e-\frac{5}{6},&k_1,k_2,k_3,k_4=0\\
-\frac{1}{\pi^2}-\frac{5}{12}\im,&k_2,k_3,k_4=0,k_1=1\\
-\frac{1}{\pi^2}+\frac{5}{12}\im,&k_2,k_3,k_4=0,k_1=-1\\
\frac{1}{\pi^2k_1^2},&k_2,k_3,k_4=0,\betr{k_1}\geq2\\
\sqrt{2}\frac{(-1)^{k_3}-1}{\pi^2k_3^2}+\sqrt{2}\frac{(-1)^k_3\e-1}{\pi^2k_3^2+1},&k_1,k_2,k_4=0,k_3\neq0\\
-\frac{5}{2\pi^2k_2^2}\im,&k_3,k_4=0,k_1=1,k_2\neq0\\
\frac{5}{2\pi^2k_2^2}\im,&k_3,k_4=0,k_1=-1,k_2\neq0\\
2\sqrt{2}\frac{(-1)^{k_3}-1}{\pi^4k_1^2k_3^2},&k_2,k_4=0,k_1,k_3\neq0\\
0,&\text{else}
\end{cases}.
\end{align*}
Furthermore, we consider the norm of the function $f_2$ and observe $\norm{f_2}_{\Lp{2}(\Dbb^{\dbf_2})}={}\frac{103}{120}+\frac{\e^2}{2}$. Using this we observe the variance $\sigma^2(f_2)=\frac{59}{360}+\frac{5}{3}\e-\frac{\e^2}{2}.$ Next, we consider the ANOVA terms $\sigma^2(f_\ubf)$ for subsets of indices $\ubf\subseteq[4]$ using their series representation $\sigma^2(f_\ubf)=\sum_{\substack{\kbf\in\Kbb^{\dbf_2}\\\supp\kbf=\ubf}}\betrk{c_\kbf^{\dbf_2}(f)}^2$. For the further calculation we need $\sum_{k=1}^\infty\frac{1}{k^4}=\frac{\pi^4}{90}$ and $\sum_{k=1}^\infty\frac{1}{(2k-1)^4}=\frac{\pi^4}{96}$. Using this we get the variances $\sigma^2(f_\ubf)$ of the ANOVA terms $f_\ubf$,
\begin{alignat*}{2}
&\sigma^2(f_{\{1\}})&&=\sum_{\substack{k=-\infty\\k\neq0}}^\infty\betrk{c_{k\ebf_1}^{\dbf_2}}^2=\frac{2}{\pi^4}+\frac{25}{72}+2\sum_{k=2}^\infty\frac{1}{\pi^4k^4}=\frac{133}{360}\\
&\sigma^2(f_{\{1,2\}})&&=\sum_{\substack{k=-\infty\\k\neq0}}^\infty\sum_{\substack{j=-\infty\\j\neq0}}^\infty\betrk{c_{k\ebf_1+j\ebf_2}^{\dbf_2}}^2=4\sum_{k=1}^\infty\frac{25}{4\pi^4k^4}=\frac{5}{18}\\
&\sigma^2(f_{\{1,3\}})&&=\sum_{\substack{k=-\infty\\k\neq0}}^\infty\sum_{j=1}^\infty\betrk{c_{k\ebf_1+j\ebf_3}^{\dbf_2}}^2=2\sum_{k=1}^\infty\sum_{j=1}^\infty\frac{8\left((-1)^j-1\right)^2}{\pi^8k^4j^4}=\frac{1}{135}.
\end{alignat*}
Since the other mixed coefficients except $c^{\dbf_2}_{k\ebf_3}(f_2)$ are zero, we get the variance $\sigma^2(f_{\{3\}})$ through the theorem of Parseval, 
\begin{align*}
\sigma^2(f_{\{3\}})&=\sigma^2(f_2)-\sigma^2(f_{\{1\}})-\sigma^2(f_{\{1,2\}})-\sigma^2(f_{\{1,3\}})=-\frac{53}{108}+\frac{5}{3}\e-\frac{\e^2}{2}.
\end{align*}
Finally, we get the analytic global sensitivity indices
\begin{alignat*}{3}
&\rho(\{1\},f_2)&&=\frac{133}{59+600\e-180\e^2}&&\approx0.369507,\\
&\rho(\{3\},f_2)&&=\frac{-530+1800\e-540\e^2}{177+1800\e-540\e^2}&&\approx0.345259,\\
&\rho(\{1,2\},f_2)&&=\frac{100}{59+600\e-180\e^2}&&\approx0.277825,\text{ and }\\
&\rho(\{1,3\},f_2)&&=\frac{8}{177+1800\e-540\e^2}&&\approx0.007409.
\end{alignat*}

\end{document}

%% file: makros.tex
\newcommand{\abf}{\mathbf{a}}
\newcommand{\bbf}{\mathbf{b}}
\newcommand{\cbf}{\mathbf{c}}
\newcommand{\dbf}{\mathbf{d}}
\newcommand{\ebf}{\mathbf{e}}
\newcommand{\fbf}{\mathbf{f}}
\newcommand{\gbf}{\mathbf{g}}
\newcommand{\hbf}{\mathbf{h}}
\newcommand{\ibf}{\mathbf{i}}
\newcommand{\jbf}{\mathbf{j}}
\newcommand{\kbf}{\mathbf{k}}
\newcommand{\lbf}{\mathbf{l}}
\newcommand{\mbf}{\mathbf{m}}
\newcommand{\nbf}{\mathbf{n}}
\newcommand{\obf}{\mathbf{o}}
\newcommand{\pbf}{\mathbf{p}}
\newcommand{\qbf}{\mathbf{q}}
\newcommand{\rbf}{\mathbf{r}}
\newcommand{\sbf}{\mathbf{s}}
\newcommand{\tbf}{\mathbf{t}}
\newcommand{\ubf}{\mathbf{u}}
\newcommand{\vbf}{\mathbf{v}}
\newcommand{\wbf}{\mathbf{w}}
\newcommand{\xbf}{\mathbf{x}}
\newcommand{\ybf}{\mathbf{y}}
\newcommand{\zbf}{\mathbf{z}}
\newcommand{\Abf}{\mathbf{A}}
\newcommand{\Bbf}{\mathbf{B}}
\newcommand{\Cbf}{\mathbf{C}}
\newcommand{\Dbf}{\mathbf{D}}
\newcommand{\Ebf}{\mathbf{E}}
\newcommand{\Fbf}{\mathbf{F}}
\newcommand{\Gbf}{\mathbf{G}}
\newcommand{\Hbf}{\mathbf{H}}
\newcommand{\Ibf}{\mathbf{I}}
\newcommand{\Jbf}{\mathbf{j}}
\newcommand{\Kbf}{\mathbf{K}}
\newcommand{\Lbf}{\mathbf{L}}
\newcommand{\Mbf}{\mathbf{M}}
\newcommand{\Nbf}{\mathbf{N}}
\newcommand{\Obf}{\mathbf{O}}
\newcommand{\Pbf}{\mathbf{P}}
\newcommand{\Qbf}{\mathbf{Q}}
\newcommand{\Rbf}{\mathbf{R}}
\newcommand{\Sbf}{\mathbf{S}}
\newcommand{\Tbf}{\mathbf{T}}
\newcommand{\Ubf}{\mathbf{U}}
\newcommand{\Vbf}{\mathbf{V}}
\newcommand{\Wbf}{\mathbf{W}}
\newcommand{\Xbf}{\mathbf{X}}
\newcommand{\Ybf}{\mathbf{Y}}
\newcommand{\Zbf}{\mathbf{Z}}
\newcommand{\nullbf}{\mathbf{0}}
\newcommand{\einsbf}{\mathbf{1}}
\newcommand{\alphabf}{\boldsymbol{\alpha}}
\newcommand{\betabf}{\boldsymbol{\beta}}
\newcommand{\gammabf}{\boldsymbol{\gamma}}
\newcommand{\deltabf}{\boldsymbol{\delta}}
\newcommand{\epsilonbf}{\boldsymbol{\epsilon}}
\newcommand{\zetabf}{\boldsymbol{\zeta}}
\newcommand{\etabf}{\boldsymbol{\eta}}
\newcommand{\thetabf}{\boldsymbol{\theta}}
\newcommand{\iotabf}{\boldsymbol{\iota}}
\newcommand{\kappabf}{\boldsymbol{\kappa}}
\newcommand{\lambdabf}{\boldsymbol{\lambda}}
\newcommand{\mubf}{\boldsymbol{\mu}}
\newcommand{\nubf}{\boldsymbol{\nu}}
\newcommand{\xibf}{\boldsymbol{\xi}}
\newcommand{\pibf}{\boldsymbol{\pi}}
\newcommand{\rhobf}{\boldsymbol{\rho}}
\newcommand{\sigmabf}{\boldsymbol{\sigma}}
\newcommand{\taubf}{\boldsymbol{\tau}}
\newcommand{\phibf}{\boldsymbol{\phi}}
\newcommand{\chibf}{\boldsymbol{\chi}}
\newcommand{\psibf}{\boldsymbol{\psi}}
\newcommand{\omegabf}{\boldsymbol{\omega}}
\newcommand{\Gammabf}{\boldsymbol{\Gamma}}
\newcommand{\Deltabf}{\boldsymbol{\Delta}}
\newcommand{\Thetabf}{\boldsymbol{\Theta}}
\newcommand{\Lambdabf}{\boldsymbol{\Lambda}}
\newcommand{\Xibf}{\boldsymbol{\Xi}}
\newcommand{\Pibf}{\boldsymbol{\Pi}}
\newcommand{\Sigmabf}{\boldsymbol{\Sigma}}
\newcommand{\Phibf}{\boldsymbol{\Phi}}
\newcommand{\Psibf}{\boldsymbol{\Psi}}
\newcommand{\Omegabf}{\boldsymbol{\Omega}}

\newcommand{\Abb}{\mathbb{A}}
\newcommand{\BBbb}{\mathbb{B}}
\newcommand{\Cbb}{\mathbb{C}}
\newcommand{\Dbb}{\mathbb{D}}
\newcommand{\Ebb}{\mathbb{E}}
\newcommand{\Fbb}{\mathbb{F}}
\newcommand{\Gbb}{\mathbb{G}}
\newcommand{\Hbb}{\mathbb{H}}
\newcommand{\Ibb}{\mathbb{I}}
\newcommand{\Jbb}{\mathbb{J}}
\newcommand{\Kbb}{\mathbb{K}}
\newcommand{\Lbb}{\mathbb{L}}
\newcommand{\Mbb}{\mathbb{M}}
\newcommand{\Nbb}{\mathbb{N}}
\newcommand{\Obb}{\mathbb{O}}
\newcommand{\Pbb}{\mathbb{P}}
\newcommand{\Qbb}{\mathbb{Q}}
\newcommand{\Rbb}{\mathbb{R}}
\newcommand{\Sbb}{\mathbb{S}}
\newcommand{\Tbb}{\mathbb{T}}
\newcommand{\Ubb}{\mathbb{U}}
\newcommand{\Vbb}{\mathbb{V}}
\newcommand{\Wbb}{\mathbb{W}}
\newcommand{\Xbb}{\mathbb{X}}
\newcommand{\Ybb}{\mathbb{Y}}
\newcommand{\Zbb}{\mathbb{Z}}

\newcommand{\arm}{\mathrm{a}}
\newcommand{\brm}{\mathrm{b}}
\newcommand{\crm}{\mathrm{c}}
\newcommand{\drm}{\mathrm{d}}
\newcommand{\erm}{\mathrm{e}}
\newcommand{\frm}{\mathrm{f}}
\newcommand{\grm}{\mathrm{g}}
\newcommand{\hrm}{\mathrm{h}}
\newcommand{\irm}{\mathrm{i}}
\newcommand{\jrm}{\mathrm{j}}
\newcommand{\krm}{\mathrm{k}}
\newcommand{\lrm}{\mathrm{l}}
\newcommand{\mrm}{\mathrm{m}}
\newcommand{\nrm}{\mathrm{n}}
\newcommand{\orm}{\mathrm{o}}
\newcommand{\prm}{\mathrm{p}}
\newcommand{\qrm}{\mathrm{q}}
\newcommand{\rrm}{\mathrm{r}}
\newcommand{\srm}{\mathrm{s}}
\newcommand{\trm}{\mathrm{t}}
\newcommand{\urm}{\mathrm{u}}
\newcommand{\vrm}{\mathrm{v}}
\newcommand{\wrm}{\mathrm{w}}
\newcommand{\xrm}{\mathrm{x}}
\newcommand{\yrm}{\mathrm{y}}
\newcommand{\zrm}{\mathrm{z}}
\newcommand{\Arm}{\mathrm{A}}
\newcommand{\Brm}{\mathrm{B}}
\newcommand{\Crm}{\mathrm{C}}
\newcommand{\Drm}{\mathrm{D}}
\newcommand{\Erm}{\mathrm{E}}
\newcommand{\Frm}{\mathrm{F}}
\newcommand{\Grm}{\mathrm{G}}
\newcommand{\Hrm}{\mathrm{H}}
\newcommand{\Irm}{\mathrm{I}}
\newcommand{\Jrm}{\mathrm{J}}
\newcommand{\Krm}{\mathrm{K}}
\newcommand{\Lrm}{\mathrm{L}}
\newcommand{\Mrm}{\mathrm{M}}
\newcommand{\Nrm}{\mathrm{N}}
\newcommand{\Orm}{\mathrm{O}}
\newcommand{\Prm}{\mathrm{P}}
\newcommand{\Qrm}{\mathrm{Q}}
\newcommand{\Rrm}{\mathrm{R}}
\newcommand{\Srm}{\mathrm{S}}
\newcommand{\Trm}{\mathrm{T}}
\newcommand{\Urm}{\mathrm{U}}
\newcommand{\Vrm}{\mathrm{V}}
\newcommand{\Wrm}{\mathrm{W}}
\newcommand{\Xrm}{\mathrm{X}}
\newcommand{\Yrm}{\mathrm{Y}}
\newcommand{\Zrm}{\mathrm{Z}}

\newcommand{\acal}{\mathcal{a}}
\newcommand{\bcal}{\mathcal{b}}
\newcommand{\ccal}{\mathcal{c}}
\newcommand{\dcal}{\mathcal{d}}
\newcommand{\ecal}{\mathcal{e}}
\newcommand{\fcal}{\mathcal{f}}
\newcommand{\gcal}{\mathcal{g}}
\newcommand{\hcal}{\mathcal{h}}
\newcommand{\ical}{\mathcal{i}}
\newcommand{\jcal}{\mathcal{j}}
\newcommand{\kcal}{\mathcal{k}}
\newcommand{\lcal}{\mathcal{l}}
\newcommand{\mcal}{\mathcal{m}}
\newcommand{\ncal}{\mathcal{n}}
\newcommand{\ocal}{\mathcal{o}}
\newcommand{\pcal}{\mathcal{p}}
\newcommand{\qcal}{\mathcal{q}}
\newcommand{\rcal}{\mathcal{r}}
\newcommand{\scal}{\mathcal{s}}
\newcommand{\tcal}{\mathcal{t}}
\newcommand{\ucal}{\mathcal{u}}
\newcommand{\vcal}{\mathcal{v}}
\newcommand{\wcal}{\mathcal{w}}
\newcommand{\xcal}{\mathcal{x}}
\newcommand{\ycal}{\mathcal{y}}
\newcommand{\zcal}{\mathcal{z}}
\newcommand{\Acal}{\mathcal{A}}
\newcommand{\Bcal}{\mathcal{B}}
\newcommand{\Ccal}{\mathcal{C}}
\newcommand{\Dcal}{\mathcal{D}}
\newcommand{\Ecal}{\mathcal{E}}
\newcommand{\Fcal}{\mathcal{F}}
\newcommand{\Gcal}{\mathcal{G}}
\newcommand{\Hcal}{\mathcal{H}}
\newcommand{\Ical}{\mathcal{I}}
\newcommand{\Jcal}{\mathcal{J}}
\newcommand{\Kcal}{\mathcal{K}}
\newcommand{\Lcal}{\mathcal{L}}
\newcommand{\Mcal}{\mathcal{M}}
\newcommand{\Ncal}{\mathcal{N}}
\newcommand{\Ocal}{\mathcal{O}}
\newcommand{\Pcal}{\mathcal{P}}
\newcommand{\Qcal}{\mathcal{Q}}
\newcommand{\Rcal}{\mathcal{R}}
\newcommand{\Scal}{\mathcal{S}}
\newcommand{\Tcal}{\mathcal{T}}
\newcommand{\Ucal}{\mathcal{U}}
\newcommand{\Vcal}{\mathcal{V}}
\newcommand{\Wcal}{\mathcal{W}}
\newcommand{\Xcal}{\mathcal{X}}
\newcommand{\Ycal}{\mathcal{Y}}
\newcommand{\Zcal}{\mathcal{Z}}

\newcommand{\und}{\text{ und }}
\newcommand{\andt}{\text{ and }}
\newcommand{\ko}{,\;}

\newcommand{\mop}[1]{\mathop{\mathrm{#1}}}
\newcommand{\supp}{\mop{supp}\limits}
\newcommand{\sign}{\mop{sign}\limits}
\newcommand{\diag}{\mop{diag}\limits}
\newcommand{\spanm}{\mop{span}\limits}
\newcommand{\img}{\mop{img}\limits}
\newcommand{\id}{\mop{id}}
\newcommand{\sinc}{\mop{sinc}}

\newcommand{\argmin}{\mathop{\arg\min}\limits}
\newcommand{\argmax}{\mathop{\arg\max}\limits}
\newcommand{\limesinf}{\mathop{\lim\inf}\limits}
\newcommand{\limessup}{\mathop{\lim\sup}\limits}
\newcommand{\esssup}{\mathop{\mop{ess}\sup}\limits}
\newcommand{\essinf}{\mathop{\mop{ess}\inf}\limits}
\newcommand{\arctann}{\mathop{\mathrm{arctan2}}\limits}

\newcommand{\igl}[3]{\int_{#1}{#2}\;\drm{#3}}
\newcommand{\iglse}[4]{\int_{#1}^{#2}{#3}\;\drm{#4}}
\newcommand{\skp}[2]{\left\langle{#1},{#2}\right\rangle}
\newcommand{\skpr}[3]{\skp{#1}{#2}_{#3}}
\newcommand{\skpk}[2]{\langle{#1},{#2}\rangle}
\newcommand{\skprk}[3]{\skpk{#1}{#2}_{#3}}
\newcommand{\Lp}[1]{\Lrm_{#1}}
\newcommand{\betr}[1]{\left\lvert{#1}\right\rvert}
\newcommand{\betrk}[1]{\lvert{#1}\rvert}
\newcommand{\norm}[1]{\left\lVert{#1}\right\rVert}
\newcommand{\normk}[1]{\lVert{#1}\rVert}
\newcommand{\men}[2]{\left\{{#1}\;\middle\vert\;{#2}\right\}}
\newcommand{\menk}[2]{\{{#1}\mid{#2}\}}
\newcommand{\fktk}[3]{{#1}\colon{#2}\rightarrow{#3}}
\newcommand{\fkt}[5]{\fktk{#1}{#2}{#3},\;{#1}({#4})\coloneqq{#5}}
\newcommand{\fktka}[3]{{#1}&\colon{#2}\rightarrow{#3}}
\newcommand{\fkta}[5]{\fktka{#1}{#2}{#3},\;{#1}({#4})\coloneqq{#5}}
\newcommand{\fktmt}[5]{\fktk{#1}{#2}{#3},\;{#4}\mapsto{#5}}
\newcommand{\fktl}[5]{&\rlap{$\fktk{#1}{#2}{#3},$}&&\\&{#1}({#4})&\;\coloneqq{}&{#5}} 
\newcommand{\fktlnl}[1]{&&={}&{#1}}
\newcommand{\fktmtl}[5]{{#1}\colon{#2}&\rightarrow{#3},\\{#4}&\mapsto{#5}}

\newcommand{\im}{\mathrm{i}}
\newcommand{\e}{\mathrm{e}}